\documentclass{article}%
\usepackage{amsfonts}
\usepackage{amssymb}
\usepackage{graphicx}
\usepackage{amsmath}%
\newtheorem{theorem}{Theorem}

\newtheorem{corollary}[theorem]{Corollary}

\newtheorem{definition}[theorem]{Definition}

\newtheorem{lemma}[theorem]{Lemma}

\newtheorem{proposition}[theorem]{Proposition}

\newenvironment{proof}[1][Proof]{\textbf{#1.} }{\ \rule{0.5em}{0.5em}}
\begin{document}

\title{On the interlace polynomials}
\author{Lorenzo Traldi\\Lafayette College\\Easton, Pennsylvania 18042}
\date{}
\maketitle

\begin{abstract}
The generating function that records the sizes of directed circuit partitions
of a connected 2-in, 2-out digraph $D$ can be determined from the
interlacement graph of $D$ with respect to a directed Euler circuit; the same
is true of\ the generating functions for other kinds of circuit partitions.
The interlace polynomials of Arratia, Bollob\'{a}s and Sorkin [J. Combin.
Theory Ser. B 92 (2004) 199-233; Combinatorica 24 (2004) 567-584] extend the
corresponding functions from interlacement graphs to arbitrary graphs. We
introduce a multivariate interlace polynomial that is an analogous extension
of a multivariate generating function for undirected circuit partitions of
undirected 4-regular graphs. The multivariate polynomial incorporates several
different interlace polynomials\ that have been studied by different authors,
and its properties include invariance under a refined version of local
complementation and a simple recursive definition.

\bigskip


\bigskip

Mathematics Subject\ Classification. 05C50

\end{abstract}

\section{Introduction}

In order to introduce our results and the background theory in a\ precise way,
we need to fix some definitions. A \emph{graph} $G$ consists of a finite set
$V(G)$ of \emph{vertices}, and a finite set $E(G)$ of \emph{edges}; each
element of $E(G)$ is incident on one or two vertices. An edge incident on only
one vertex is a \emph{loop}. Two distinct vertices incident on a single edge
are \emph{neighbors}; the set of neighbors of a vertex $v$ is the \emph{open
neighborhood} $N_{G}(v)$. It is often convenient to think of an edge as
consisting of two distinct \emph{half-edges}, each of which is incident on
precisely one vertex. An edge is \emph{directed }by specifying that one
half-edge is initial and the other is terminal; as the half-edges are
distinct, every edge can be directed in two different ways. The \emph{degree}
of a vertex $v$ is the number of half-edges incident at $v$. A $k$%
\emph{-regular} graph is one whose vertices all have degree $k$. Edges
incident on precisely the same vertices are \emph{parallel}, and a graph with
no loops and no parallels is \emph{simple}. A \emph{circuit} in a graph is a
sequence $v_{1}$, $h_{1}$, $h_{2}^{\prime}$, $v_{2}$, ..., $v_{k}$, $h_{k}$,
$h_{k+1}^{\prime}$, $v_{k+1}=v_{1}$ such that for each $i$, $h_{i}$ and
$h_{i}^{\prime}$ are half-edges incident on $v_{i}$, and $h_{i}$ and
$h_{i+1}^{\prime}$ are half-edges of a single edge $e_{i}$. A vertex may
appear repeatedly on a circuit, but an edge may not appear more than once. If
it happens that for every $i$, $e_{i}$ is a directed edge with initial
half-edge $h_{i}$, then the circuit is \emph{directed};\ in general a directed
graph may contain both directed circuits and undirected circuits. An
\emph{Eulerian graph }is a graph that possesses at least one \emph{Eulerian
circuit}, i.e., a circuit which includes every edge.

This paper concerns a family of graph invariants, the \emph{interlace
polynomials}. We use the term in a generic sense, to include also some
polynomials that were introduced under other names. All of these polynomials
are motivated by the circuit theory of 4-regular graphs.%
\begin{figure}
[pt]
\begin{center}
\includegraphics[
trim=2.405358in 8.786434in 1.743204in 1.140289in,
height=0.6036in,
width=3.1012in
]%
{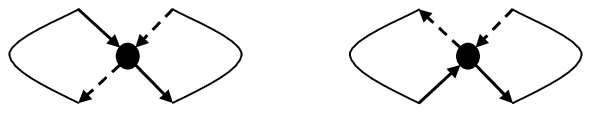}%
\caption{$C\ast v$ is obtained by reversing either of the two $v$-to-$v$ walks
within the incident circuit of $C$.}%
\label{dirintf}%
\end{center}
\end{figure}

Four cornerstones of this theory were laid in the 1960s and 1970s. A connected
4-regular graph is Eulerian, of course. More generally, an arbitrary 4-regular
graph has \textit{Euler systems}, each of which contains one Euler circuit for
each connected component of the graph. Kotzig \cite{K} introduced the $\kappa
$\textit{-transformations}: if $C$ is an Euler system of a 4-regular graph $F$
and $v\in V(F)$ then the $\kappa$-transform $C\ast v$ is the Euler system
obtained from $C$ by reversing one of the two $v$-to-$v$ walks within the
circuit of $C$ incident on~$v$. \textit{Kotzig's theorem} is the first of the
four cornerstones; it tells us that all the Euler systems of $F$ can be
obtained from any one using $\kappa$-transformations.

Although our discussion is focused on 4-regular graphs, we should certainly
mention that Kotzig's theorem extends to arbitrary Eulerian graphs; see
Fleischner's books \cite{F1, F2} for an account of the general theory.

The second cornerstone of the circuit theory of 4-regular graphs is the
\textit{interlacement graph} $\mathcal{I}(F,C)$ of a 4-regular graph with
respect to an\ Euler system $C$. $\mathcal{I}(F,C)$ is the simple graph with
the same vertices as $F$, in which two vertices $v$ and $w$ are neighbors if
and only if they are \textit{interlaced} with respect to $C$, i.e., they
appear in the order $v...w...v...w$ on one of the circuits of $C$. The graphs
that arise as interlacement graphs are called \textit{circle graphs}. (This
definition is usually restricted to Euler circuits of connected 4-regular
graphs, but the restriction would be inconvenient here because there are
natural ways to recursively simplify 4-regular graphs, which sometimes
disconnect them.) This construction was discussed by Bouchet \cite{Bold} and
Read and Rosenstiehl \cite{RR}, who observed that the relationship between
$\mathcal{I}(F,C)$ and $\mathcal{I}(F,C\ast v)$ is described by \textit{simple
local complementation} at $v$: if $v\neq x\neq y\neq v$ and $x,y$ are both
neighbors of $v$ in $F$ then $x$ and $y$ are adjacent in $\mathcal{I}(F,C\ast
v)$ if and only if they are not adjacent in $\mathcal{I}(F,C)$. Later, Bouchet
introduced isotropic systems to study circle graphs and the equivalence
relation on arbitrary graphs generated by simple local complementations
\cite{Bi1, Bec, Bi2, Bco}.

By the way, we use the term \textit{simple local complementation} to
distinguish this operation from the one that Arratia, Bollob\'{a}s and Sorkin
called \textit{local complementation} in \cite{A1, A2, A}; that operation also
includes loop-toggling at neighbors of $v$.

If $C$ is an\ Euler system of $F$, then $F$ is made into a 2-in, 2-out digraph
$D$ by choosing either of the two orientations for each circuit of $C$, and
directing the edges of $F$ accordingly. If $v$ and $w$ are neighbors in
$\mathcal{I}(F,C)$ then the iterated $\kappa$-transform $C\ast v\ast w\ast v$
is also a directed Euler system for $D$, obtained by interchanging the two
$v$-to-$w$ walks within the incident circuit of $C$. Following \cite{A1, A2},
we refer to the operation $C\mapsto C\ast v\ast w\ast v$ as
\textit{transposition}; the induced operation on interlacement graphs is
\textit{pivoting}, denoted\textit{ }$G\mapsto G^{vw}$. Kotzig \cite{K},
Pevzner \cite{Pev} and Ukkonen \cite{U} proved that transpositions suffice to
obtain all the directed Euler systems for a 2-in, 2-out digraph from any one;
a more general form of this theorem was proven by Fleischner, Sabidussi, and
Wenger \cite{F}.%

\begin{figure}
[ptb]
\begin{center}
\includegraphics[
trim=1.201442in 6.685563in 1.071980in 0.937048in,
height=2.3341in,
width=4.51in
]%
{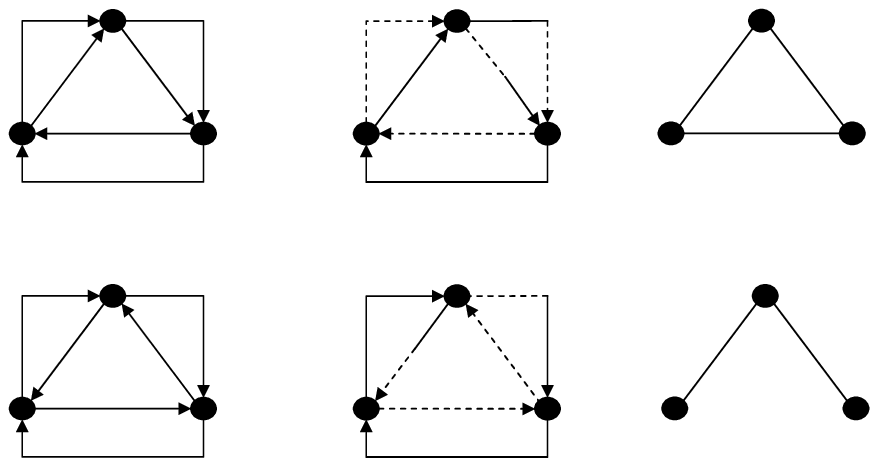}%
\caption{$D_{1}$ (top left) and $D_{2}$ (bottom left) are nonisomorphic
directed graphs whose undirected versions are isomorphic. They yield
interlacement graphs that are equivalent under local complementation but not
under pivoting.}%
\label{dirintf1}%
\end{center}
\end{figure}

As examples of these notions, consider the 2-in, 2-out digraphs $D_{1}$ and
$D_{2}$ of Figure \ref{dirintf1}. They are small enough so that each has only
one directed Euler circuit, up to automorphism. Two Euler circuits are
indicated in the figure; to trace an Euler circuit follow the directed edges,
making sure to maintain the same dash pattern when traversing a vertex. (The
dash pattern may be changed while traversing an edge.) The corresponding
interlacement graphs are indicated in the figure's third column. Pivoting on
an edge in the lower interlacement graph produces an isomorphic replica, with
a different degree-2 vertex; pivoting on an edge in the upper interlacement
graph has no effect at all. The fact that the two interlacement graphs are not
equivalent under pivoting reflects the fact that $D_{1}$ and $D_{2}$ are not
isomorphic. On the other hand, simple local complementation at the single
degree-2 vertex of the lower interlacement graph produces the upper
interlacement graph, reflecting the fact that the undirected versions of
$D_{1}$ and $D_{2}$ are isomorphic.

Let $F$ be a 4-regular graph with $c(F)$ connected components, and let $C$ be
an Euler system of $F$. A \textit{circuit partition} or \textit{Eulerian
partition} of $F$ is a partition of $E(F)$ into edge-disjoint circuits. Such a
partition is determined by choosing, at each vertex of $F$, one of the
three\ \textit{transitions} (pairings of the incident half-edges): the
transition that appears in the incident circuit of $C$, which we label $\phi$,
for \textquotedblleft follow\textquotedblright; the other transition
consistent with the edge-directions given by the incident circuit of $C$,
which we label $\chi$, for \textquotedblleft cross\textquotedblright; or the
transition that is inconsistent with these edge-directions, which we label
$\psi$. See Figure \ref{intext1}. (We should mention that we use the
terminology of Ellis-Monaghan and Sarmiento \cite{E} and Jaeger \cite{J}, in
which a transition at $v$ specifies both pairings of incident half-edges that
might appear in a circuit partition. Other authors, including Bouchet and
Kotzig, use \textquotedblleft transition\textquotedblright\ in a slightly
different way, to refer to a single pairing of half-edges, and require a
separate matching-up of the pairings.) If $n=\left\vert V(F)\right\vert $ then
$F$ has $3^{n}$ circuit partitions, given by choosing one of the three
transitions at each vertex. A 2-in, 2-out digraph has $2^{n}$ directed circuit
partitions.%
\begin{figure}
[tb]
\begin{center}
\includegraphics[
trim=1.064559in 8.155316in 2.008726in 1.348879in,
height=0.9236in,
width=3.9055in
]%
{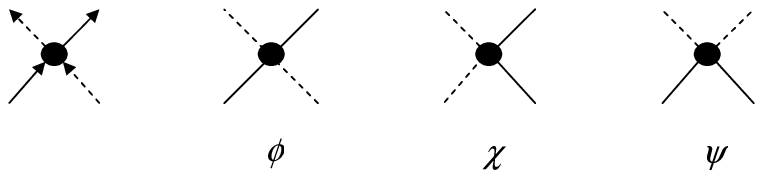}%
\caption{An orientation of the circuit of $C$ incident at $v$ is indicated at
left. Each circuit partition of $F$ involves one of the three pictured
transitions at $v$. Both orientations of the incident circuit of $C$ result in
the same labeling of these three transitions.}%
\label{intext1}%
\end{center}
\end{figure}

The third cornerstone of the circuit theory of 4-regular graphs is the idea of
defining a polynomial invariant of a 2-in, 2-out digraph (or 4-regular graph)
by using some form of the generating function $\sum x^{\left\vert P\right\vert
}$ that records the sizes of (un)directed circuit partitions. This idea was
introduced by Las Vergnas \cite{L}, who observed that a polynomial defined
recursively by Martin \cite{Ma} is essentially equivalent to the generating
function. Las Vergnas also extended the idea to general Eulerian graphs, and
to 4-regular graphs that arise as medial graphs imbedded in the projective
plane and the torus \cite{L2, L1}. In particular, if a 4-regular graph $F$ is
imbedded in the plane then its complementary regions can be colored
checkerboard fashion, yielding a pair of dual graphs with $F$ as medial, and
the Tutte polynomial of either of the two dual graphs yields the directed
circuit partition generating function for a certain directed version of $F$;
this theorem foreshadowed the famous connection between the Tutte polynomial
and the Jones polynomial of knot theory \cite{Jone, Kau, Th}. We will not
focus any further attention on imbedded graphs in this paper; we refer the
interested reader to Ellis-Monaghan and Moffatt \cite{EMM} for a thorough
discussion including recent results.

Martin observed that the circuit partition generating functions can be
described recursively. Suppose $F$ is a 4-regular graph with an Euler system
$C$, $D$ is the 2-in, 2-out digraph corresponding to a choice of orientations
for the circuits of $C$, and $v$ is unlooped in $F$. The directed circuit
partitions of $D$ fall into two classes, those that follow $C$ through $v$ and
those that involve the $\chi$ transition at $v$. These two classes correspond
to directed circuit partitions of the two digraphs $D_{\phi}$ and $D_{\chi}$
obtained by \textit{directed detachment} at $v$, illustrated in Figure
\ref{dirintf2}. (The term \textit{detachment} was coined by Nash-Williams; see
\cite{N1} for instance.)%
\begin{figure}
[ptb]
\begin{center}
\includegraphics[
trim=1.335852in 8.576775in 1.205565in 0.941327in,
height=0.9115in,
width=4.305in
]%
{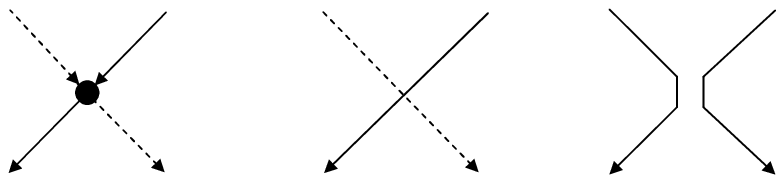}%
\caption{The two digraphs obtained by directed detachment at $v$. }%
\label{dirintf2}%
\end{center}
\end{figure}

Similarly, the circuit partitions of $F$ fall into three classes according to
the transitions at $v$, and the three classes correspond to circuit partitions
of three graphs $F_{\phi}$, $F_{\chi}$ and $F_{\psi}$ obtained by detachment
at $v$.

The fourth cornerstone involves an equality due to Cohn and Lempel \cite{CL}.
In its original form, the equality relates the number of cycles in a
permutation to the nullity of an associated skew-symmetric matrix over
$GF(2)$. An equivalent form of the equality relates the number of circuits in
a directed circuit partition of a connected 2-in, 2-out digraph $D$ to the
nullity of the adjacency matrix of an associated subgraph of an interlacement
graph $\mathcal{I}(D,C)$. It is remarkable that versions of this useful
equality have been discovered and rediscovered by combinatorialists and
topologists so many times \cite{Be, BM, Bu, Br, J1, Jo, KR, Lau, MP, Me, M, R,
So, S, Z}. The Cohn-Lempel equality extends to a \textit{circuit-nullity
formula} for undirected circuit partitions in undirected 4-regular graphs. We
simply state the formula here, and refer to \cite{Tbn, Tnew} for more detailed
accounts. Let $G=\mathcal{I}(F,C)$, let $P$ be a circuit partition of $F$, and
let $G_{P}=\mathcal{I}_{P}(F,C)$ be the graph obtained from $G$ by removing
each vertex at which $P$ involves the $\phi$ transition used by $C$, and
attaching a loop at each vertex where $P$ involves the $\psi$
transition.\ Then the circuit-nullity formula states that%
\[
\left\vert P\right\vert -c(F)=\nu(G_{P}),
\]
where $\nu(G_{P})$ denotes the $GF(2)$-nullity of the adjacency matrix of
$G_{P}$, i.e., the difference between $\left\vert V(G_{P})\right\vert $ and
the $GF(2)$-rank. (The adjacency matrix of $G_{P}$ is the $V(G_{P})\times
V(G_{P})$ matrix over $GF(2)$ in which a diagonal entry is nonzero if and only
if the corresponding vertex is looped, and an off-diagonal entry is nonzero if
and only if the two corresponding vertices are neighbors.) The original
Cohn-Lempel equality is essentially the special case in which $c(F)=1$ and no
$\psi$ transition appears in $P$.%
\begin{figure}
[ptb]
\begin{center}
\includegraphics[
trim=1.068682in 3.075359in 0.938395in 1.070760in,
height=4.9441in,
width=4.7072in
]%
{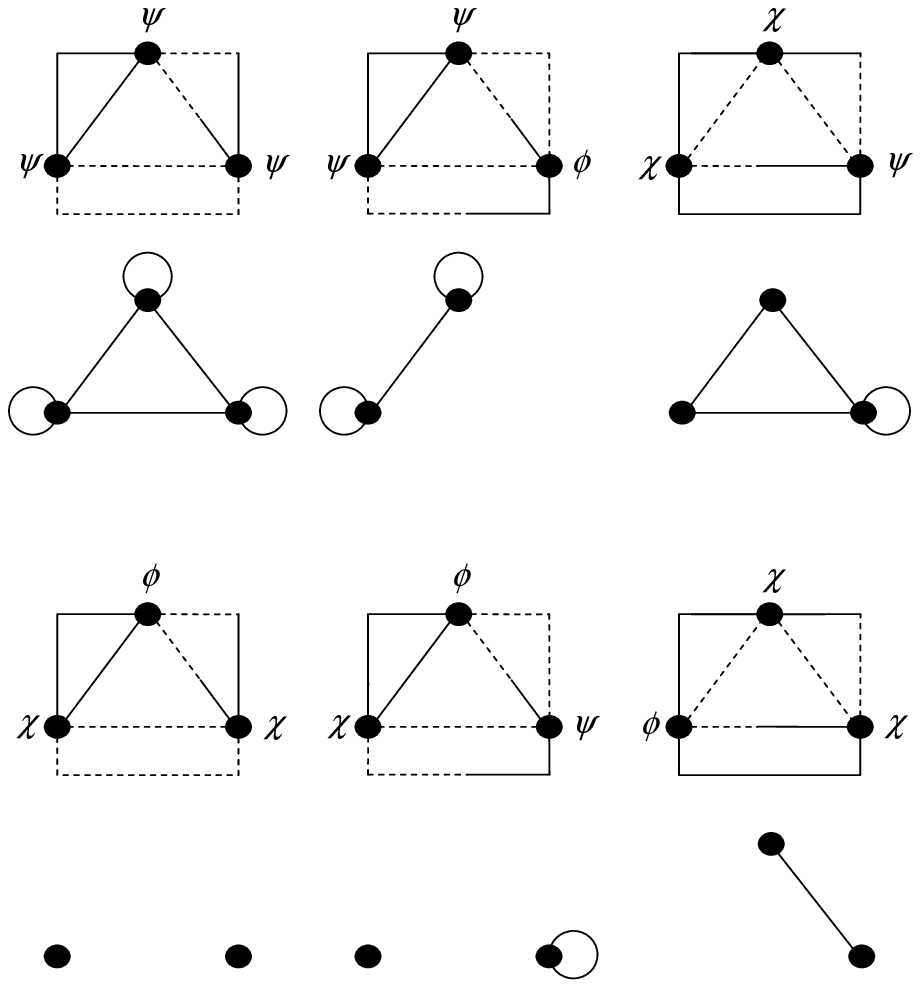}%
\caption{Three circuit partitions $P_{1},P_{2},P_{3}$ in the 4-regular graph
of Figure \ref{dirintf1}, with the transition labels and $G_{P}$ graphs
corresponding to the two Euler circuits shown there.}%
\label{dirintf3}%
\end{center}
\end{figure}

Some examples appear in Figure \ref{dirintf3}. Three circuit partitions
$P_{1},P_{2},P_{3}$ of the undirected version of the graphs $D_{1}$ and
$D_{2}$ of Figure \ref{dirintf1} are depicted in the top row. Each circuit is
traced out by maintaining the dash pattern when traversing a vertex; the dash
pattern may change when traversing an edge, though. Transition labels indicate
the relationships between these circuit partitions and the Euler circuit of
$D_{1}$ indicated in Figure \ref{dirintf1}. The corresponding graphs
$G_{P_{i}}$ appear in the second row. In the third row of Figure
\ref{dirintf3} we see the same three circuit partitions, now with transition
labels that indicate their relationships with the Euler circuit of $D_{2}$
indicated in Figure \ref{dirintf1}. The corresponding graphs $G_{P_{i}}$
appear in the fourth row. The circuit-nullity formula is satisfied because
both sets of graphs $G_{P_{i}}$ satisfy $\nu(G_{P_{1}})=2=3-1$, $\nu(G_{P_{2}%
})=1=2-1$ and $\nu(G_{P_{3}})=0=1-1$.

If $F$ is a 4-regular graph with an Euler system $C$, then the circuit-nullity
formula provides a bijective equivalence between these two information-sets.

\begin{enumerate}
\item List the circuit partitions of $F$; for each circuit partition, specify
the corresponding choice of the $\phi,\chi$ or $\psi$ transition at each
vertex, and also the number of circuits.

\item List the looped full subgraphs of $\mathcal{I}(F,C)$; for each looped
full subgraph, specify the corresponding decision to remove, retain or loop
each vertex, and also the $GF(2)$-nullity of the adjacency matrix.
\end{enumerate}

This bijective equivalence tells us that the generating functions that record
the sizes of (un)directed circuit partitions in $F$ -- or equivalently, the
Martin polynomials -- can be reformulated as generating functions that record
the binary nullities of adjacency matrices of (looped) full subgraphs of
$\mathcal{I}(F,C)$. Observe that there is no reason the nullity-based
reformulations of these generating functions should be restricted to
interlacement graphs; the definitions extend unchanged to arbitrary simple
graphs. The graph polynomial that extends the directed Martin polynomial, the
\emph{(vertex-nullity) interlace polynomial} $q_{N}(G)$, was introduced by
Arratia, Bollob\'{a}s and Sorkin \cite{A1, A2}. Their original definition
extended the recursive description of the Martin polynomial, rather than the
nullity-based reformulation; they derived the nullity-based form in \cite{A},
where they also introduced a two-variable interlace polynomial $q(G)$.
Subsequently, Aigner and van der Holst \cite{AH} defined another interlace
polynomial $Q(G)$, which extends the reformulated version of the undirected
Martin polynomial. More recently, Courcelle \cite{C} introduced a multivariate
interlace polynomial $C(G)$, which extends a logically equivalent form of the
information described in item 2 above. These graph polynomials are related in
various ways to the (un)restricted \textquotedblleft Tutte-Martin
polynomials\textquotedblright\ of isotropic systems studied by\ Bouchet
\cite{Bi3, B5}.

Our purpose in the present paper is to incorporate the $\phi,\chi,\psi$ labels
into the machinery outlined above. This is accomplished in two stages. In
Section 2, $\phi,\chi$ and $\psi$ are introduced into the circuit theory of
4-regular graphs as transition labels. These transition labels are neither
absolute nor arbitrary; they are defined with respect to a particular\ Euler
system, and they are modified in particular ways when $\kappa$-transformations
are applied to that\ Euler system. The labeled versions of the Martin
polynomials are generating functions that record the sizes of (un)directed
circuit partitions, along with the corresponding transition labels. (The
$\psi$ label is not needed for directed circuit partitions.) These labeled
versions of the Martin polynomials are specializations of the \emph{transition
polynomials} discussed by Jaeger \cite{J} and Ellis-Monaghan and Sarmiento
\cite{E}; the transition polynomials also incorporate transition labels, but
the labels are arbitrary and consequently carry less information than
$\phi,\chi$ and $\psi$.

The equivalence between items 1 and 2 above tells us that the labeled versions
of the (un)directed Martin polynomials can be reformulated as labeled
generating functions that record the $GF(2)$-nullities of adjacency matrices
of (looped) full subgraphs of interlacement graphs. In Sections 3 -- 5, these
generating functions are extended from labeled interlacement graphs to general
labeled graphs, with vertex labels $\phi,\chi,\psi$ now representing three
ways to treat a vertex (remove it, retain it, or attach a loop) rather than
three ways to choose a transition. The result is to unify all the polynomial
invariants of graphs and isotropic systems discussed in the paragraph before
last in a single multivariate graph polynomial, the \emph{labeled interlace
polynomial} $Q_{\lambda}(G)$, whose properties include a three-term recursive
definition and invariance under labeled local complementation. Setting
$\psi\equiv0$ in $Q_{\lambda}(G)$ we obtain the 2-label interlace polynomial
$q_{\lambda}(G)$, which extends the nullity-based reformulation of the
directed Martin polynomial. It satisfies a two-term recursion and is invariant
under labeled pivoting.

In\ Section 6 we detail the relationships between $Q_{\lambda}$ and the
several kinds of interlace polynomials that have been studied since the
original definition of Arratia, Bollob\'{a}s and Sorkin \cite{A1, A2, A}.
In\ Section 7 we present formulas for the labeled interlace polynomials of
graphs with split decompositions, and discuss their computational significance.

\textbf{Acknowledgments}. Before proceeding to present these notions in
detail, we should thank D. P. Ilyutko, V. O. Manturov and L. Zulli. The idea
of using label-switching local complementations was inspired by many
conversations with them while we studied the use of interlacement to describe
the Jones polynomial and Kauffman bracket of a link diagram \cite{I, IM, IM1,
T2, T3, T4, TZ}. Preliminary drafts of the paper were significantly improved
by the kind advice of B. Bollob\'{a}s, B. Courcelle, J. A. Ellis-Monaghan, C.
Hoffmann, M. Las Vergnas and an anonymous reader.

\section{Transition labels and circuit partitions}

\begin{definition}
\label{one}An Euler system $C$ in a 4-regular graph $F$ is \emph{labeled} by
giving a trio of functions $\phi_{C}$, $\chi_{C}$, $\psi_{C}$ mapping $V(F)$
into some commutative ring $R$.
\end{definition}

When we want to specify the ring in question, we refer to $C$ as
$R$\emph{-labeled}. We will see in Section 6 that it is useful to consider a
variety of labeling strategies, in a variety of rings. However, there is an
especially natural\ way to implement Definition \ref{one}. Suppose $R$ is a
polynomial ring with $3\cdot\left\vert V(F)\right\vert $ indeterminates, one
indeterminate corresponding to each transition at each vertex of $F$. For each
$v\in V(F)$, define the images of $v$ under $\phi_{C}$, $\chi_{C}$, and
$\psi_{C}$ in accordance with Figure \ref{intext1}. Then the label function
$\phi_{C}$ actually specifies the Euler system $C$.

Recall that a $\kappa$-transformation is applied by reversing one of the two
$v$-to-$v$ walks in an Euler system, as depicted in Figure \ref{dirintf}.
Clearly this reversal affects some transition labels: at $v$ itself, the
$\phi$ and $\psi$ labels are interchanged; and at a vertex $w$ that appears
precisely once on the reversed $v$-to-$v$ walk, the $\chi$ and $\psi$ labels
are interchanged.

\begin{definition}
\label{three}Let $v$ be a vertex of a 4-regular graph $F$, and let $C$ be a
labeled\ Euler system of $F$. The \emph{labeled }$\kappa$\emph{-transform}
$C\ast v$ is obtained by reversing one of the two $v$-to-$v$ walks within the
circuit of $C$ incident on $v$, and making the following label changes:
$\phi_{C\ast v}(v)=\psi_{C}(v)$, $\psi_{C\ast v}(v)=\phi_{C}(v)$, and for each
$w$ that neighbors $v$ in $\mathcal{I}(F,C)$, $\chi_{C\ast v}(w)=\psi_{C}(w)$
and $\psi_{C\ast v}(w)=\chi_{C}(w)$.
\end{definition}

\begin{definition}
\label{four}Two labeled Euler systems of a 4-regular graph are $\kappa
$\emph{-equivalent} if and only if one can be obtained from the other through
labeled $\kappa$-transformations.
\end{definition}

Kotzig's theorem \cite{K} tells us that $\kappa$-transformations can be used
to obtain all the Euler systems of $F$ from any one; it follows that if $F$
has Euler systems $C_{1}$ and $C_{2}$ then each labeled version of $C_{1}$ is
$\kappa$-equivalent to a unique labeled version of $C_{2}$.

\begin{definition}
\label{two}Let $C$ be an $R$-labeled Euler system of $F$, and suppose $y\in
R$. For each circuit partition $P$ of $F$, let $\phi(P,C)$, $\chi(P,C)$ and
$\psi(P,C)$ denote the sets of vertices of $F$ where $P$ involves the
transition labeled $\phi$, $\chi$ or $\psi$ (respectively) with respect to
$C$. The $R$\emph{-labeled circuit partition generating function of }$F$
\emph{with respect to }$C$ is
\[
\pi(F,C)=\sum_{P\in\mathcal{P}(F)}\left(
{\displaystyle\prod_{v\in\phi(P,C)}}
\phi_{C}(v)\right)  \left(
{\displaystyle\prod_{v\in\chi(P,C)}}
\chi_{C}(v)\right)  \left(
{\displaystyle\prod_{v\in\psi(P,C)}}
\psi_{C}(v)\right)  y^{\left\vert P\right\vert -c(F)},
\]
where $\mathcal{P}(F)$ is the set of circuit partitions of $F$ and $c(F)$ is
the number of connected components of $F$.
\end{definition}

Different systems of labels yield generating functions with different levels
of detail. If the labels take the natural values in the polynomial ring
$\mathbb{Z}[\{y\}\cup\{\phi_{C}(v),\chi_{C}(v),\psi_{C}(v)$
$\vert$
$v\in V(F)\}]$ with $1+3\cdot\left\vert V(F)\right\vert $ independent
indeterminates, then $\pi(F,C)\,$\ is essentially a table that lists, for
every circuit partition $P$, $\left\vert P\right\vert -c(F)$ along with the
relationship between $C$ and $P$ at every vertex. Similarly, the $\psi\equiv0$
specialization of this polynomial is essentially a generating function for
directed circuit partitions of a 2-in, 2-out digraph $D$ obtained by directing
the edges of $F$ according to orientations of the circuits of $C$. The theory
of this polynomial is outlined in Section 5. If $\phi,\chi,\psi\equiv
1\in\mathbb{Z}[y]$ then $\pi(F,C)$ is the ordinary (unlabeled) generating
function for circuit partitions of $F$, and if $\psi\equiv0$ and $\phi
,\chi\equiv1\in\mathbb{Z}[y]$ then $\pi(F,C)$ is the ordinary generating
function for circuit partitions of $D$.

\begin{proposition}
\label{pinvariance}Let $C_{1}$ and $C_{2}$ be labeled Euler systems in $F$. If
$C_{1}$ and $C_{2}$ are $\kappa$-equivalent, then $\pi(F,C_{1})=\pi(F,C_{2})$.
\end{proposition}

\begin{proof}
Labeled $\kappa$-transformations preserve $\pi$ term by term, because the
contribution of each $P\in\mathcal{P}(F)$ is unchanged.
\end{proof}

The labeled circuit partition generating function $\pi(F,C)$ incorporates more
precise information about the structure of a 4-regular graph $F$ than other
transition-based polynomials that have appeared in the literature. For
instance, the \textquotedblleft Tutte-Martin polynomials\textquotedblright%
\ discussed by Bouchet \cite{Bi3} involve \textquotedblleft
coding\textquotedblright\ each vertex of $F$ by choosing an arbitrary labeling
of the transitions. (Indeed Proposition (5.2) of \cite{Bi1} states that the
isotropic systems associated to two coded 4-regular graphs are isomorphic if
and only if they can be obtained from differently coded versions of the same
4-regular graph.) Jaeger's transition polynomial \cite{J} and the generalized
transition polynomial of Ellis-Monaghan and Sarmiento \cite{EMS} also involve
arbitrary transition labels. In contrast, the label functions $\phi_{C}$,
$\chi_{C}$ and $\psi_{C}$ are not arbitrary: they are associated in special
ways with the positioning of $C$ within $F$, and they are handled in special
ways by labeled $\kappa$-transformations.

In particular, it is not generally possible to simply transpose two labels at
one vertex using labeled $\kappa$-transformations. For example, suppose $C$ is
any labeled Euler circuit of the graph $F$ in\ Figure \ref{dirintf5}, with 21
independent indeterminates serving as transition labels (three for each
vertex). Then $\pi(F,C)$ determines the $\chi$ labels at the two central
vertices of the graph: they are the only transition labels that appear only in
terms divisible by $y$, because the corresponding transitions are the only
ones that do not appear in any Euler circuit. Similarly, if $F_{1}$ is the
undirected version of the graph $D_{1}$ of Figure \ref{dirintf1}, and $C_{1}$
is the Euler circuit of $D_{1}$ indicated in Figure \ref{dirintf1}, then the
$\psi$ labels in $F_{1}$ are distinguished by the fact that they are the only
ones that appear in a term of $\pi(F_{1},C_{1})$ divisible by $y^{3}$.%

\begin{figure}
[ptb]
\begin{center}
\includegraphics[
trim=2.941348in 8.834570in 2.679950in 0.803337in,
height=0.8224in,
width=1.9977in
]%
{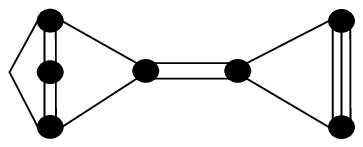}%
\caption{All\ Euler circuits have the same $\chi$ transitions at the central
vertices.}%
\label{dirintf5}%
\end{center}
\end{figure}

The labeled circuit partition generating function $\pi(F,C)$ satisfies a
labeled version of the detachment-based recursion mentioned in the
introduction. Suppose $C$ is a labeled\ Euler system of $F$ and $v$ is an
unlooped vertex of $F$. (We leave the consideration of looped vertices to the
reader.) Then as discussed in the introduction, there are three associated
4-regular graphs obtained by detachment at $V$, denoted $F_{\phi}$, $F_{\chi}$
and $F_{\psi}$ according to the transitions that define the detachments. As
illustrated in Figure \ref{dirintf9}, $C$ yields labeled Euler systems in all
three detachments. $F_{\phi}$ has a labeled Euler system $C_{\phi}$ whose
circuits simply follow the circuits of $C$, omitting $v$. $F_{\psi}$ has a
labeled Euler system $C_{\psi}$, obtained in the same way from the labeled
$\kappa$-transform $C\ast v$. The situation in $F_{\chi}$ is more complicated,
as there are two distinct cases. If $v$ is not interlaced with any other
vertex with respect to $C$, then $c(F_{\chi})=c(F)+1$ and $F_{\chi}$ has a
labeled\ Euler system $C_{\chi}$ obtained by separating the two $v$-to-$v$
circuits within the incident circuit of $C$. On the other hand if $v$ is
interlaced with a vertex $w$, then $c(F_{\chi})=c(F)$ and a labeled\ Euler
system $C_{\chi}$ of $F_{\chi}$ is obtained by following the circuits of
$C\ast w\ast v\ast w$, omitting $v$.%

\begin{figure}
[ptb]
\begin{center}
\includegraphics[
trim=1.070331in 3.743915in 0.676172in 1.340322in,
height=4.235in,
width=4.9009in
]%
{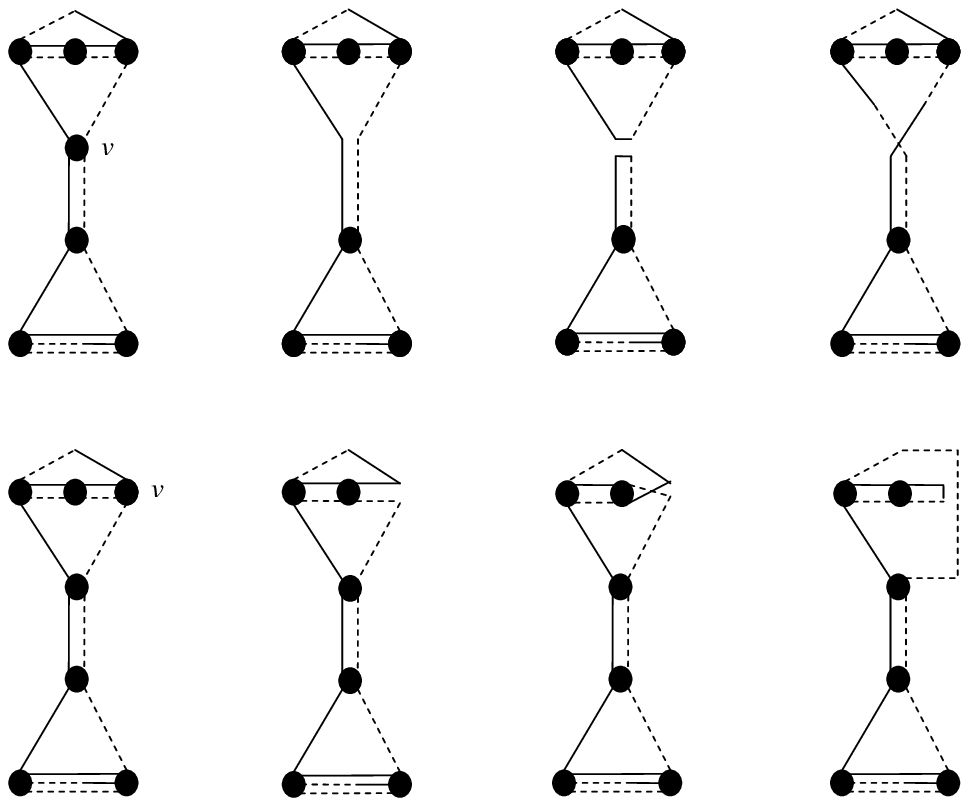}%
\caption{The graph of Figure \ref{dirintf5} appears on the left, with an Euler
circuit indicated by dashes. The top row depicts the three detachments at a
central vertex and the bottom row depicts the three detachments at another
vertex.}%
\label{dirintf9}%
\end{center}
\end{figure}

\begin{theorem}
\label{rec}Suppose $v$ is an unlooped vertex of $F$. Then%
\[
\pi(F,C)=\phi_{C}(v)\cdot\pi(F_{\phi},C_{\phi})+\chi_{C}(v)\cdot\pi(F_{\chi
},C_{\chi})+\psi_{C}(v)\cdot\pi(F_{\psi},C_{\psi}).
\]

\end{theorem}

\begin{proof}
The formula is justified in the natural way, by classifying the circuit
partitions of $F$ according to the transitions at $v$. The notation $(F_{\chi
},C_{\chi})$ is ambiguous -- it does not tell us whether the first case or the
second case holds in $F$, and in the second case it does not reflect the fact
that different choices of $w$ will yield different labeled Euler systems in
$F_{\chi}$ -- but this ambiguity does no harm because the formula holds for
every choice of $C_{\chi}$.
\end{proof}

\section{Vertex labels and partitions}

As discussed in the introduction, the circuit-nullity formula allows us to
reformulate the theory described in\ Section 2: rather than thinking of
circuit partitions in a 4-regular graph $F$, we may think of graphs obtained
from an interlacement graph $\mathcal{I}(F,C)$ by removing some vertices and
looping some vertices. This reformulated version of the theory extends
directly from interlacement graphs to arbitrary simple graphs.

\begin{definition}
\label{label}A simple graph $G$ is ($R$\emph{-})\emph{labeled} by giving a
trio of functions $\phi_{G}$, $\chi_{G}$, $\psi_{G}$ mapping $V(G)$ into some
commutative ring $R$.
\end{definition}

Once again, the most natural ring to use as $R$ is a polynomial ring with
three independent indeterminates $\phi_{G}(v)$, $\chi_{G}(v)$ and $\psi
_{G}(v)$ for each $v\in V(G)$.

\begin{definition}
\label{llc}Let $G$ be a labeled\ simple graph with a vertex $v$. The
\emph{labeled local complement} $G_{\lambda}^{v}$ is the labeled graph
obtained from $G$ by toggling adjacencies between distinct neighbors of $v$
and making the following label changes: $\phi_{G_{\lambda}^{v}}(v)=\psi
_{G}(v)$, $\psi_{G_{\lambda}^{v}}(v)=\phi_{G}(v)$, and for each $w$ that
neighbors $v$ in $G$, $\chi_{G_{\lambda}^{v}}(w)=\psi_{G}(w)$ and
$\psi_{G_{\lambda}^{v}}(w)=\chi_{G}(w)$.
\end{definition}

Recall Definition \ref{three}: if $C$ is a labeled Euler system in a 4-regular
graph $F$, then for $v\in V(F)$ the labeled $\kappa$-transform $C\ast v$ is
obtained by reversing one of the two $v$-to-$v$ walks within the circuit of
$C$ incident on $v$, and making appropriate adjustments to the label
functions. The effect of reversing a $v$-to-$v$ walk is to toggle
interlacements between vertices that appear precisely once on that walk, i.e.,
vertices that are interlaced with $v$. We conclude the following.

\begin{theorem}
Let $F$ be a 4-regular graph with a labeled Euler system $C$, and consider
$\mathcal{I}(F,C)$ and $\mathcal{I}(F,C\ast v)$ as labeled graphs with the
trios of label functions $\phi_{C}$, $\chi_{C}$, $\psi_{C}$ and $\phi_{C\ast
v}$, $\chi_{C\ast v}$, $\psi_{C\ast v}$ (respectively). Then $\mathcal{I}%
(F,C\ast v)=\mathcal{I}(F,C)_{\lambda}^{v}$.
\end{theorem}

We now have three different kinds of local complementation: simple local
complementation, for which we use no particular symbol; the local
complementation $G^{v}$ used by Arratia, Bollob\'{a}s and Sorkin \cite{A1, A2,
A}, which combines simple local complementation with loop-toggling at
neighbors of $v$; and labeled local complementation. The loop-toggling at
neighbors of $v$ in $G^{v}$ and the $\chi\leftrightarrow\psi$ exchange at
neighbors of $v$ in $G_{\lambda}^{v}$ are essentially the same thing; what is
special about labeled local complementation is the $\phi\leftrightarrow\psi$
exchange at $v$ itself. As we show in Theorems \ref{invariance} and
\ref{invariance1} below, this detail allows us to formulate a simple
invariance property for a rather complicated-seeming graph polynomial that
determines all the different interlace polynomials studied in \cite{AH, A1,
A2, A, Bi3, B5, C}. It is the absence of this detail that creates the seeming
lack of simple invariance properties in the original discussions of some of
these graph polynomials.

The reformulated version\ of a circuit partition is a certain kind of vertex partition.

\begin{definition}
\label{lpart}A\emph{ labeled partition} of a labeled simple graph $G$ is a
partition $P$ of $V(G)$ into three pairwise disjoint subsets, $V(G)=\phi
(P)\cup\chi(P)\cup\psi(P)$. The set of all labeled partitions of $G$ is
denoted $\mathcal{P}_{\lambda}(G)$.
\end{definition}

Note the different uses of the term \textquotedblleft
labeled\textquotedblright\ in Definitions \ref{label} and \ref{lpart}: a graph
is labeled by specifying functions $\phi_{G}$, $\chi_{G}$ , $\psi
_{G}:V(G)\rightarrow R$, but a partition $P$ of $V(G)$ into three disjoint
subsets is labeled by identifying one of the subsets as $\phi(P)$, another as
$\chi(P)$, and the third as $\psi(P)$.

\begin{definition}
If $v\in V(G)$ then the \emph{labeled local complement} of $P\in
\mathcal{P}_{\lambda}(G)$ is the labeled partition $P_{\lambda}^{v}%
\in\mathcal{P}(G_{\lambda}^{v})$ obtained from $P$ by making the following
changes: $v\in\phi(P_{\lambda}^{v})$ if and only if $v\in\psi(P)$, $v\in
\psi(P_{\lambda}^{v})$ if and only if $v\in\phi(P)$, and if $w$ is a neighbor
of $v$ in $G$ then $w\in\chi(P_{\lambda}^{v})$ if and only if $w\in\psi(P)$,
and $w\in\psi(P_{\lambda}^{v})$ if and only if $w\in\chi(P)$.
\end{definition}

The circuit-nullity formula suggests the following.

\begin{definition}
\label{assograph}If $P$ is a labeled partition of a simple graph $G$ then
$G_{P}$ denotes the graph obtained from $G$ by removing every vertex in
$\phi(P)$ and attaching a loop at every vertex in $\psi(P)$.
\end{definition}

If $P$ is a circuit partition of a 4-regular graph $F$ with two Euler systems
$C_{1}$ and $C_{2}$, then the circuit-nullity formula tells us that
$\left\vert P\right\vert $ is related to the binary nullities of the adjacency
matrices of the two $G_{P}$ graphs obtained from $\mathcal{I}(F,C_{1})$ and
$\mathcal{I}(F,C_{2})$. $P$ itself does not change when we change Euler
system, so these two adjacency matrices must have the same nullity. It may be
a surprise that this invariance of nullity extends to arbitrary graphs, even
though there is no fixed object that plays the role of $P$.

\begin{theorem}
\label{invariance}If $v\in V(G)$ then for every $P$ in $\mathcal{P}_{\lambda
}(G)$, we have $\nu(G_{P})=\nu((G_{\lambda}^{v})_{P_{\lambda}^{v}})$.
\end{theorem}

\begin{proof}
If $v\in\phi(P)$, the theorem states that
\[
\left(
\begin{array}
[c]{ccc}%
M_{11} & M_{12} & M_{13}\\
M_{21} & M_{22} & M_{23}\\
M_{31} & M_{32} & M_{33}%
\end{array}
\right)  \text{ and }\left(
\begin{array}
[c]{cccc}%
1 & \mathbf{1} & \mathbf{1} & \mathbf{0}\\
\mathbf{1} & \overline{M}_{11} & \overline{M}_{12} & M_{13}\\
\mathbf{1} & \overline{M}_{21} & \overline{M}_{22} & M_{23}\\
\mathbf{0} & M_{31} & M_{32} & M_{33}%
\end{array}
\right)
\]
have the same $GF(2)$-nullity. Here the left-hand matrix is partitioned into
sets of rows and columns corresponding respectively to the neighbors of $v$ in
$\chi(P)$, the neighbors of $v$ in $\psi(P)$, and the vertices in $\chi
(P)\cup\psi(P)$ that are not neighbors of $v$; the first row and column of the
right-hand matrix correspond to $v$. Bold numerals denote rows and columns
with all entries the same, and an overbar indicates the toggling of all
entries in a matrix over $GF(2)$. The nullity equality is verified by
observing that adding the first row to every row in the first two sets of rows
in the right-hand matrix yields%
\[
\left(
\begin{array}
[c]{cccc}%
1 & \mathbf{1} & \mathbf{1} & \mathbf{0}\\
\mathbf{0} & M_{11} & M_{12} & M_{13}\\
\mathbf{0} & M_{21} & M_{22} & M_{23}\\
\mathbf{0} & M_{31} & M_{32} & M_{33}%
\end{array}
\right)  ,
\]
whose nullity is the same is that of the first matrix displayed above.

If $v\in\psi(P)$ the preceding argument is simply reversed. That is, we use
row operations to show that
\[
\left(
\begin{array}
[c]{cccc}%
1 & \mathbf{1} & \mathbf{1} & \mathbf{0}\\
\mathbf{1} & M_{11} & M_{12} & M_{13}\\
\mathbf{1} & M_{21} & M_{22} & M_{23}\\
\mathbf{0} & M_{31} & M_{32} & M_{33}%
\end{array}
\right)  \text{ and }\left(
\begin{array}
[c]{ccc}%
\overline{M}_{11} & \overline{M}_{12} & M_{13}\\
\overline{M}_{21} & \overline{M}_{22} & M_{23}\\
M_{31} & M_{32} & M_{33}%
\end{array}
\right)
\]
have the same $GF(2)$-nullity.

If $v\in\chi(P)$ then the theorem states that
\[
\left(
\begin{array}
[c]{cccc}%
0 & \mathbf{1} & \mathbf{1} & \mathbf{0}\\
\mathbf{1} & M_{11} & M_{12} & M_{13}\\
\mathbf{1} & M_{21} & M_{22} & M_{23}\\
\mathbf{0} & M_{31} & M_{32} & M_{33}%
\end{array}
\right)  \text{ and }\left(
\begin{array}
[c]{cccc}%
0 & \mathbf{1} & \mathbf{1} & \mathbf{0}\\
\mathbf{1} & \overline{M}_{11} & \overline{M}_{12} & M_{13}\\
\mathbf{1} & \overline{M}_{21} & \overline{M}_{22} & M_{23}\\
\mathbf{0} & M_{31} & M_{32} & M_{33}%
\end{array}
\right)  \text{ }%
\]
have the same $GF(2)$-nullity. This is verified by adding the first row to
those in the second and third sets.
\end{proof}

Here is the reformulated version of Definition \ref{two}.

\begin{definition}
\label{linterlace}Let $G$ be an $R$-labeled simple graph, and suppose $y\in
R$. The \emph{labeled interlace polynomial} of $G$ is the sum%
\[
Q_{\lambda}(G)=\sum_{P\in\mathcal{P}_{\lambda}(G)}\left(
{\displaystyle\prod_{v\in\phi(P)}}
\phi_{G}(v)\right)  \left(
{\displaystyle\prod_{v\in\chi(P)}}
\chi_{G}(v)\right)  \left(
{\displaystyle\prod_{v\in\psi(P)}}
\psi_{G}(v)\right)  y^{\nu(G_{P})}.
\]

\end{definition}

Note that we call $Q_{\lambda}$ a polynomial even though $y$ is an arbitrary
element of $R$. This is a mere formality, as we could specify that $y$ be an
indeterminate and then obtain other instances of the definition through evaluation.

\begin{theorem}
\label{interp}Let $F$ be a 4-regular graph with a labeled Euler system $C$,
and let $\mathcal{I}(F,C)$ be the corresponding labeled interlacement graph.
Then%
\[
\pi(F,C)=Q_{\lambda}(\mathcal{I}(F,C)).
\]

\end{theorem}

As we mentioned in the introduction, the labeled interlace polynomial yields
all the different kinds of interlace polynomials in the literature, by using
different label values. This might suggest that the properties of $Q_{\lambda
}$ would be more complicated than those of the other polynomials; instead the
theory of $Q_{\lambda}$ turns out to be considerably simpler.

\begin{theorem}
\label{invariance1}Let $G$ be a labeled simple graph with a vertex $v$. Then
$Q_{\lambda}(G)=Q_{\lambda}(G_{\lambda}^{v})$.
\end{theorem}

\begin{proof}
This follows immediately from Theorem \ref{invariance} and the fact that for
every $P\in\mathcal{P}_{\lambda}(G)$,%
\begin{align*}
&  \left(
{\displaystyle\prod_{v\in\phi(P)}}
\phi_{G}(v)\right)  \left(
{\displaystyle\prod_{v\in\chi(P)}}
\chi_{G}(v)\right)  \left(
{\displaystyle\prod_{v\in\psi(P)}}
\psi_{G}(v)\right) \\
&  =\left(
{\displaystyle\prod_{v\in\phi(P_{\lambda}^{v})}}
\phi_{G_{\lambda}^{v}}(v)\right)  \left(
{\displaystyle\prod_{v\in\chi(P_{\lambda}^{v})}}
\chi_{G_{\lambda}^{v}}(v)\right)  \left(
{\displaystyle\prod_{v\in\psi(P_{\lambda}^{v})}}
\psi_{G_{\lambda}^{v}}(v)\right)  .
\end{align*}

\end{proof}

How can it be that $Q_{\lambda}$ is invariant under labeled local
complementation, when choosing particular labels in $Q_{\lambda}$ yields the
multivariable interlace polynomials of \cite{A} and \cite{C}, which seem to
have no invariance properties at all? The answer is given in Section 6 below:
the choices of labels that yield these polynomials do not maintain the
separation of $\phi$, $\chi$ and $\psi$. Losing this three-fold distinction
makes it impossible to apply Definition \ref{llc}, and consequently the
properties of $Q_{\lambda}$ are obscured.

\begin{theorem}
\label{recursion}The labeled interlace polynomial of a labeled simple graph is
recursively determined by these three properties.

\begin{enumerate}
\item If $G$ consists only of a single vertex $v$ then
\[
Q_{\lambda}(G)=\phi_{G}(v)+\chi_{G}(v)\cdot y+\psi_{G}(v).
\]

\item If $G_{1}$ and $G_{2}$ are disjoint graphs then $Q_{\lambda}(G_{1}\cup
G_{2})=Q_{\lambda}(G_{1})\cdot Q_{\lambda}(G_{2})$.

\item Suppose $v$ and $w$ are neighbors in a labeled simple graph $G$. Then%
\[
Q_{\lambda}(G)=\phi_{G}(v)\cdot Q_{\lambda}(G-v)+\chi_{G}(v)\cdot Q_{\lambda
}(((G_{\lambda}^{w})_{\lambda}^{v})-v)+\psi_{G}(v)\cdot Q_{\lambda
}((G_{\lambda}^{v})-v).
\]

\end{enumerate}
\end{theorem}

\begin{proof}
The first property is a special case of Definition \ref{linterlace}. The
second follows from the fact that a labeled partition $P$ of $G_{1}\cup G_{2}$
is simply the union of labeled partitions $P_{1},P_{2}$ of $G_{1}$ and $G_{2}$
(respectively); consequently the adjacency matrix of $G_{P}$ is%
\[
\left(
\begin{array}
[c]{cc}%
A_{1} & 0\\
0 & A_{2}%
\end{array}
\right)  ,
\]
where $A_{1}$ and $A_{2}$ are the adjacency matrices of $(G_{1})_{P_{1}}$ and
$(G_{2})_{P_{2}}$.

The three summands in the recursive formula correspond to the natural
partition $\mathcal{P}_{\lambda}(G)=S_{1}\cup S_{2}\cup S_{3}$, with $S_{1}$
containing the partitions that have $v\in\phi(P)$, $S_{2}$ containing those
that have $v\in\chi(P)$, and $S_{3}$ containing those that have $v\in\psi(P)$.

Consider the bijection $f:S_{1}\rightarrow\mathcal{P}_{\lambda}(G-v)$ given by
$\psi(f(P))=\psi(P)$, $\chi(f(P))=\chi(P)$ and $\phi(f(P))=\phi(P)-\{v\}$.
Definition \ref{assograph} tells us that $(G-v)_{f(P)}=G_{P}$ for every $P\in
S_{1}$, so\ $\phi_{G}(v)\cdot Q_{\lambda}(G-v)$ equals the sum of the
contributions to $Q_{\lambda}(G)$ of the partitions in $S_{1}$. This justifies
the first summand.

According to Theorems \ref{invariance} and \ref{invariance1}, if $P\in S_{3}$
then the contribution of $P$ to $Q_{\lambda}(G)$ equals the contribution of
$P_{\lambda}^{v}$ to $Q_{\lambda}(G_{\lambda}^{v})$. Also, $P\in S_{3}$ if and
only if $v\in\phi(P_{\lambda}^{v})$, so the argument just given for $S_{1}$
applies to $\{P_{\lambda}^{v}$
$\vert$
$P\in S_{3}\}$. Consequently $\phi_{G_{\lambda}^{v}}(v)\cdot Q_{\lambda
}((G_{\lambda}^{v})-v)$ equals the sum of the contributions to $Q_{\lambda
}(G_{\lambda}^{v})$ of the partitions $P_{\lambda}^{v}$ with $P\in S_{3}$; as
$\phi_{G_{\lambda}^{v}}(v)=\psi_{G}(v)$, it follows that $\psi_{G}(v)\cdot
Q_{\lambda}((G_{\lambda}^{v})-v)$ equals the sum of the contributions to
$Q_{\lambda}(G)$ of the partitions $P\in S_{3}$. This justifies the third summand.

The second summand is justified in a similar way. If $P\in\mathcal{P}%
_{\lambda}(G)$ then $v\in\chi(P)$ if and only if $v\in\psi(P_{\lambda}^{w})$,
and this holds if and only if $v\in\phi((P_{\lambda}^{w})_{\lambda}^{v})$;
hence the argument given for $S_{1}$ applies to $\{(P_{\lambda}^{w})_{\lambda
}^{v}$
$\vert$
$P\in S_{2}\}$.
\end{proof}

Unlike Theorem \ref{rec}, Theorem \ref{recursion} provides a complete
recursion. The difference is that looped vertices in 4-regular graphs, which
are not covered in Theorem \ref{rec}, give rise to isolated vertices in
interlacement graphs, which are covered in Theorem \ref{recursion}.

The double local complementation in the second summand of part 3 of Theorem
\ref{recursion} is equivalent to pivoting; this version of the recursion is
useful in proving Theorem \ref{qrecursion} below.

\begin{proposition}
\label{recursion1}The formula of part 3 of Theorem \ref{recursion} may be
rewritten as%
\[
Q_{\lambda}(G)=\phi_{G}(v)\cdot Q_{\lambda}(G-v)+\chi_{G}(v)\cdot Q_{\lambda
}((((G_{\lambda}^{w})_{\lambda}^{v})_{\lambda}^{w})-v)+\psi_{G}(v)\cdot
Q_{\lambda}((G_{\lambda}^{v})-v).
\]

\end{proposition}

\begin{proof}
$(((G_{\lambda}^{w})_{\lambda}^{v})-v)_{\lambda}^{w}=(((G_{\lambda}%
^{w})_{\lambda}^{v})_{\lambda}^{w})-v$, so Theorem \ref{invariance1} tells us
that the two formulas are the same.
\end{proof}

\section{Looped graphs}

Looped vertices play two very different roles in the theory discussed in the
introduction. A 4-regular graph $F$ may certainly have looped vertices, but
interlacement graphs may not; they are simple by definition. Looped vertices
reappear in the circuit-nullity formula, in association with $\psi$
transitions. Similarly, the definition of $Q_{\lambda}(G)$ presumes that $G$
is a labeled simple graph, but looped vertices play an important role because
the graph $G_{P}$ associated to a labeled partition $P$ has a loop at each
$v\in\psi(P)$.

Observing that the difference between a looped vertex and an unlooped vertex
in $G_{P}$ is the difference between $\chi(P)$ and $\psi(P)$, we are led to a
natural way to extend the theory of Section 3 to labeled, looped simple graphs.

\begin{definition}
\label{assoloop}If a labeled graph $G$ is simple except for some looped
vertices, then its \emph{simplification} is the labeled simple graph
$G_{simp}$ obtained by interchanging $\chi(v)$ and $\psi(v)$ at each looped
vertex, and then removing all loops.
\end{definition}

The discussion of Section 3 is applied to a labeled, looped simple graph
$G$\ indirectly, by using $G_{simp}$ as a stand-in for $G$. Note that with
this approach looped graphs do not add anything new, so the difference between
restricting\ the theory to simple graphs and extending\ the theory to looped
graphs is essentially a matter of style, not substance. We choose to present
the restricted\ theory because it is (appropriately) simpler. The
extended\ theory requires more complicated statements of definitions and
theorems -- for instance, the extended version of Definition \ref{llc} would
involve different label-swaps at looped vertices, and the extended version of
Theorem \ref{recursion} would require an extra step to eliminate loops -- and
the complications seem unnecessary because all they amount to is the repeated
application of Definition \ref{assoloop}.

\section{Labeled pivoting\label{lpiv}}

Recall that if $v$ is a vertex of a graph $G$ then the \emph{open
neighborhood} of $v$ in $G$ is $N_{G}(v)$ $=$ $\{a\in V(G)$
$\vert$
$a$ neighbors $v$ in $G\}$.

\begin{proposition}
\label{lpivot}Let $v$ and $w$ be neighbors in a labeled simple graph $G$. Then
$((G_{\lambda}^{w})_{\lambda}^{v})_{\lambda}^{w}=((G_{\lambda}^{v})_{\lambda
}^{w})_{\lambda}^{v}=G^{\prime}$ is the labeled simple graph obtained from $G$
by making the following changes:

\begin{enumerate}
\item $\phi_{G^{\prime}}(v)=\chi_{G}(v)$, $\chi_{G^{\prime}}(v)=\phi_{G}%
(v)\,$, $\phi_{G^{\prime}}(w)=\chi_{G}(w)$ and $\chi_{G^{\prime}}(w)=\phi
_{G}(w)$.

\item Toggle the adjacency status of each pair of distinct vertices $a,b\in
V(G)-\{v,w\}$ such that $a\in N_{G}(v)$, $b\in N_{G}(w)$ and at least one of
$a,b\,$\ is not in $N_{G}(v)\cap N_{G}(w)$.

\item Exchange the neighbors of $v$ and $w$ (other than $v$ and $w$ themselves).
\end{enumerate}
\end{proposition}

\begin{proof}
This follows directly from Definition \ref{llc}.
\end{proof}

The operation $G\mapsto G^{\prime}$ is \textit{labeled pivoting} on the edge
$vw$; we use the notation $G^{\prime}=G_{\lambda}^{vw}$. Arratia, Bollob\'{a}s
and Sorkin \cite{A1, A2} noted that for unlabeled simple graphs, the result is
the same up to isomorphism if the neighbor-exchange of step 3 is replaced by a
\textquotedblleft label swap\textquotedblright\ in which the names of $v$ and
$w$ are exchanged. An analogue of their observation holds here too: the result
is the same up to isomorphism if steps 1 and 3 of Proposition \ref{lpivot} are
replaced by the following swap of labels at $v$ and $w$: $\phi_{G^{\prime}%
}(v)=\chi_{G}(w)$, $\chi_{G^{\prime}}(v)=\phi_{G}(w)\,$, $\psi_{G^{\prime}%
}(v)=\psi_{G}(w)$, $\phi_{G^{\prime}}(w)=\chi_{G}(v)$, $\chi_{G^{\prime}%
}(w)=\phi_{G}(v)$ and $\psi_{G^{\prime}}(w)=\psi_{G}(v)$.

\subsection{2-in, 2-out digraphs}

As mentioned in the introduction, the transposition operation $C\mapsto C\ast
w\ast v\ast w$ (where $v$ and $w$ are interlaced on $C$) plays the directed
version of the role played by $\kappa$-transformation for undirected 4-regular
graphs: if $D$ is a 2-in, 2-out digraph then every directed Euler system of
$D$ can be obtained from any one through transpositions. Consequently pivoting
plays the same role in the theory of directed interlacement as local
complementation plays in the theory of undirected interlacement. With this
idea in mind, it is easy to formulate the following directed versions of the
definitions and results discussed in\ Sections 2 and 3. We leave the proofs to
the reader.

\begin{theorem}
\label{dirint}Let $D$ be a 2-in, 2-out digraph with a directed Euler system
$C$ and let $\mathcal{I}(D,C)$ be the corresponding labeled interlacement
graph. If $v$ and $w$ are neighbors in $\mathcal{I}(D,C)$ then $\mathcal{I}%
(D,C\ast w\ast v\ast w)=\mathcal{I}(D,C)_{\lambda}^{vw}$.
\end{theorem}

\begin{definition}
\label{dirpart}Let $D$ be a 2-in, 2-out digraph, and let $\mathcal{P}(D)$
denote the set of directed circuit partitions of $D$. If $C$ is a labeled,
directed Euler system of $D$ then for each $P\in\mathcal{P}(D)$, let $\phi
_{C}(P)$ and $\chi_{C}(P)$ denote the sets of vertices of $D$ where $P$
involves the transition labeled $\phi$ or $\chi$ (respectively) with respect
to $C$. The \emph{labeled directed circuit partition generating function of
}$D$ \emph{with respect to }$C$ is
\[
\pi(D,C)=\sum_{P\in\mathcal{P}(D)}\left(
{\displaystyle\prod_{v\in\phi(P)}}
\phi_{C}(v)\right)  \left(
{\displaystyle\prod_{v\in\chi(P)}}
\chi_{C}(v)\right)  y^{\left\vert P\right\vert -c(D)},
\]
where $c(D)$ is the number of connected components of $D$.
\end{definition}

\begin{definition}
Let $G$ be a labeled simple graph, and let $\mathcal{P}_{\lambda0}%
(G)=\{P\in\mathcal{P}_{\lambda}(G)$
$\vert$
$\psi(P)=\emptyset\}$. The \emph{2-label interlace polynomial} of $G$ is%
\[
q_{\lambda}(G)=\sum_{P\in\mathcal{P}_{\lambda0}(G)}\left(
{\displaystyle\prod_{v\in\phi(P)}}
\phi_{G}(v)\right)  \left(
{\displaystyle\prod_{v\in\chi(P)}}
\chi_{G}(v)\right)  y^{\nu(G_{P})}.
\]

\end{definition}

\begin{theorem}
\label{qinvariance}Suppose $v$ and $w$ are neighbors in a labeled simple graph
$G$. Then $q_{\lambda}(G)=q_{\lambda}(G_{\lambda}^{vw})$.
\end{theorem}

\begin{theorem}
\label{direul}Let $D$ be a 2-in, 2-out digraph with a directed Euler system
$C$ and let $\mathcal{I}(D,C)$ be the corresponding labeled interlacement
graph. Then%
\[
q_{\lambda}(\mathcal{I}(D,C))=\pi(D,C).
\]

\end{theorem}

\begin{theorem}
\label{qrecursion}The 2-label interlace polynomial of a labeled simple graph
is recursively determined by these three properties.

\begin{enumerate}
\item If $G$ consists only of a single vertex $v$ then
\[
q_{\lambda}(G)=\phi_{G}(v)+\chi_{G}(v)\cdot y.
\]

\item If $G_{1}$ and $G_{2}$ are disjoint graphs then $q_{\lambda}(G_{1}\cup
G_{2})=q_{\lambda}(G_{1})\cdot q_{\lambda}(G_{2})$.

\item Suppose $v$ and $w$ are neighbors in a labeled simple graph $G$. Then%
\[
q_{\lambda}(G)=\phi_{G}(v)\cdot q_{\lambda}(G-v)+\chi_{G}(v)\cdot q_{\lambda
}((G_{\lambda}^{vw})-v).
\]

\end{enumerate}
\end{theorem}

The above results indicate that the entire theory of Section 2 can be
restricted from 4-regular graphs to 2-in, 2-out digraphs, using directed Euler
systems to describe directed circuit partitions via interlacement, and using
labeled transposition and pivoting rather than labeled $\kappa$-transformation
and local complementation. This restriction is quite natural, but there is
also a purely algebraic way to restrict attention to directed circuit
partitions, using interlacement with respect to arbitrary (undirected) Euler
systems. The idea is to use the label functions to remove from consideration
those transitions that are inconsistent with the edge-directions.

\begin{definition}
Let $D$ be a 2-in, 2-out digraph whose undirected version is $F$, and let $C$
be a labeled Euler system of $F$. The label functions $\phi_{C},\chi_{C}%
,\psi_{C}$ are \emph{consistent with }$D$ if they have this property: For
every vertex $v\in V(F)$, the transition at $v$ that is inconsistent with the
edge-directions of $D$ corresponds to a label value $\phi_{C}(v)$, $\chi
_{C}(v)$, or $\psi_{C}(v)$ that equals $0$.
\end{definition}

The set of $D$-consistent labeled Euler systems of $F$ is closed under labeled
$\kappa$-transformations, so $D$-consistent Euler systems provide a way to
restrict attention to directed circuit partitions without any need to formally
restrict the combinatorial machinery of Section 2. Of course the
$D$-consistent versions of formulas like those of Definition \ref{two} and
Theorem \ref{rec} are simpler than they appear to be, because some of the
label values are $0$.

\subsection{$T$-compatible circuit partitions and Euler systems}

In this subsection we briefly discuss a notion introduced by Kotzig \cite{K,
K1} and subsequently investigated by other researchers, including Fleischner,
Sabidussi and Wenger \cite{F}, and Genest \cite{G1, G}.

Suppose $F$ is 4-regular and $T$ is a set that includes no more than one
transition at each vertex of $F$. The circuit partitions and Euler systems of
$F$ that avoid using the transitions from $T$ are called $T$\emph{-compatible}%
; Kotzig proved that $T$-compatible Euler systems exist for every $T$. For
example, if $D$ is a 2-in, 2-out digraph based on $F$ then using
$T=\{$transitions inconsistent with the edge-directions of $D\}$ has the
effect of restricting attention to directed circuits of $D$. Also, if we are
given a (classical or virtual) link diagram, then the \emph{Kauffman states}
\cite{Kau, Kv} of the diagram are obtained by using $T=\{$transitions
corresponding to the strands of the link components incident at the crossings
of the diagram$\}$.

Just as there are two ways to restrict the machinery of Section 2 to 2-in,
2-out digraphs, there are two ways to restrict the machinery to $T$-compatible
circuits. The purely algebraic way involves the use of arbitrary Euler systems
and $T$-compatible label functions, i.e., label functions with $0$ values
corresponding to transitions in $T$. As the set of such labeled Euler systems
is closed under labeled $\kappa$-transformations, the full combinatorial
machinery of Section 2 applies.

The more thoroughgoing restriction involves considering only Euler systems
that are themselves $T$-compatible, i.e., they do not involve any transition
from $T$. The centerpiece of this restriction is the following $T$-compatible
version of Kotzig's theorem, due to Fleischner, Sabidussi and Wenger \cite{F}.

\begin{theorem}
\label{FSW}All the $T$-compatible Euler systems of $F$ can be obtained from
any one using these three types of operations:

\begin{enumerate}
\item $\kappa$-transformations $C\mapsto C\ast v$, where $T$ does not include
a transition at $v$

\item transpositions $C\mapsto C\ast v\ast w\ast v$, where $T$ includes the
$\psi$ transitions of $C$ at $v$ and $w$, and $v$ and $w$ are interlaced on
$C$

\item $\kappa$-transformations $C\mapsto C\ast v$, where $T$ includes the
$\chi$ transition of $C$ at $v$
\end{enumerate}
\end{theorem}

Genest \cite{G1, G} codes the interlacement graph $\mathcal{I}(F,C)$ by
coloring black the vertices that appear in part 2 of Theorem \ref{FSW}, and
coloring white the vertices of part 3. Then the $\kappa$-transformations and
transpositions of Theorem \ref{FSW} yield local complementations and pivotings
that include color-swaps where appropriate. As discussed in the next section,
the Arratia-Bollob\'{a}s-Sorkin interlace polynomials of looped graphs are
motivated by a different convention, involving the attachment of loops to the
vertices of $\mathcal{I}(F,C)$ that appear in part 3 of Theorem \ref{FSW}.

\section{Interlace polynomials}

Arratia, Bollob\'{a}s and Sorkin introduced the interlace polynomials after
studying special properties of Euler circuits of 2-in, 2-out digraphs useful
in analyzing DNA\ sequencing. The original interlace polynomial \cite{A1, A2}
is a one-variable polynomial associated to a simple graph; following \cite{A},
we denote its extension to looped graphs $q_{N}(G)$, and call it the
\textit{vertex-nullity interlace polynomial}. This polynomial was first
defined recursively, using the local complementation and pivoting. Using the
recursive definition, Arratia, Bollob\'{a}s and Sorkin proved that if $G$ is
the interlacement graph of a 2-in, 2-out digraph $D$ then $q_{N}(G)$ is
essentially the generating function that records the sizes of the partitions
of $E(D)$ into directed circuits. We refer to this fact as \textit{the
fundamental interpretation of} $q_{N}$ \textit{for circle graphs}.

In \cite{A}, Arratia, Bollob\'{a}s and Sorkin showed that $q_{N}$ also has a
non-recursive definition involving the nullities of matrices over the
two-element field $GF(2)$. For $S\subseteq V(G)$ let $G[S]$ denote the full
subgraph of $G$ induced by $S$, and let $\nu(G[S])$ denote the nullity of the
adjacency matrix of $G[S]$ over $GF(2)$.

\begin{definition}
\label{ninterlace}The \emph{vertex-nullity interlace polynomial} of $G$ is%
\[
q_{N}(G)=\sum_{S\subseteq V(G)}(y-1)^{\nu(G[S])}.
\]

\end{definition}

Considering Definition \ref{assograph}, we see that $q_{N}$ is obtained from
$Q_{\lambda}$ by replacing $y$ with $y-1$, using $\phi,\chi\equiv1$ and
$\psi\equiv0$ at unlooped vertices, and using $\phi,\psi\equiv1$ and
$\chi\equiv0$ at looped vertices. The basic theory of $q_{N}$ follows from
this observation and the results of Section 3. (The $\phi,\chi\equiv1$
specialization of the polynomial $q_{\lambda}$ of Section 5 yields the
restriction of $q_{N}$ to simple graphs.) In particular, Theorem \ref{interp}
yields a fundamental interpretation of $q_{N}$ for looped circle graphs: Let
$F$ be a 4-regular graph with an Euler system $C$, and suppose $G$ is obtained
from $\mathcal{I}(F,C)$ by attaching loops at the vertices that appear in a
certain subset $\mathcal{L}\subseteq V(G)$. For each circuit partition
$P\in\mathcal{P}(F)$, let $\phi(P,C)$, $\chi(P,C)$ and $\psi(P,C)$ denote the
sets of vertices of $F$ where $P$ involves the transition labeled $\phi$,
$\chi$ or $\psi$ (respectively) with respect to $C$. Let $\mathcal{P}%
(F,\mathcal{L})$ denote the set of circuit partitions $P$ with $\chi
(P,C)\cap\mathcal{L}=\emptyset$ and $\psi(P,C)\subseteq\mathcal{L}$. (That is,
$\mathcal{P}(F,\mathcal{L})$ contains the $T$-compatible circuit partitions,
where $T$ includes the $\chi$ transitions of $C$ at vertices in $\mathcal{L}$
and the $\psi$ transitions of $C$ at vertices not in $\mathcal{L}$.) Then
\[
q_{N}(G)=\sum_{P\in\mathcal{P}(F,\mathcal{L})}\left(  y-1\right)  ^{\left\vert
P\right\vert -c(F)},
\]
where $c(F)$ is the number of connected components of $F$.

Theorem \ref{invariance1} and Proposition \ref{lpivot} imply Remark 18 of
\cite{A2}: if $v$ and $w$ are unlooped neighbors then $q_{N}(G)=q_{N}(G^{vw}%
)$. This equality does not extend to pivoting involving looped neighbors,
because part 1 of Proposition \ref{lpivot} is not compatible with the special
label values used to obtain $q_{N}$ from $Q_{\lambda}$. For the same reason,
Theorem \ref{invariance1} does not yield a useful invariance property for
$q_{N}$ under local complementation at unlooped vertices. At looped vertices
$\phi\equiv\psi$, though, so Theorem \ref{invariance1} yields the equality
$q_{N}(G)=q_{N}(G^{v})$. Observe that according to the fundamental
interpretation, if $G$ is a looped circle graph then the equalities
$q_{N}(G)=q_{N}(G^{vw})$ (for unlooped neighbors $v,w)$ and $q_{N}%
(G)=q_{N}(G^{v})$ (for looped $v$) follow immediately from Theorem \ref{FSW},
the $T$-compatible version of Kotzig's theorem due to Fleischner, Sabidussi
and Wenger \cite{F}. In general, these equalities indicate a connection
between $q_{N}$ and\ a well-known matrix operation, the \emph{principal
pivot}; detailed discussions are given by Brijder and Hoogeboom \cite{BrH,
BH2} and Glantz and Pelillo \cite{GP}.

The two-variable version of the interlace polynomial was introduced in
\cite{A}:

\begin{definition}
\label{interlace}The \emph{interlace polynomial} of a graph $G$ is
\[
q(G)=\sum_{S\subseteq V(G)}(x-1)^{\left\vert S\right\vert -\nu(G[S])}%
(y-1)^{\nu(G[S])}.
\]

\end{definition}

No fundamental interpretation was given for $q$ in \cite{A}, but rewriting
Definition \ref{interlace} as%
\[
q(G)=\sum_{S\subseteq V(G)}(x-1)^{\left\vert S\right\vert }\cdot\left(
\frac{y-1}{x-1}\right)  ^{\nu(G[S])}%
\]
we see that $q$ is obtained from the labeled interlace polynomial $Q_{\lambda
}$ by using $\phi\equiv1$, $\chi\equiv x-1$ and $\psi\equiv0$ at unlooped
vertices, using $\phi\equiv1$, $\chi\equiv0$ and $\psi\equiv x-1$ at looped
vertices, and replacing $y$ with $(y-1)/(x-1)$. Consequently a fundamental
interpretation of $q$ for looped circle graphs follows immediately from
Theorem \ref{interp}: Let $F$ be a 4-regular graph with an Euler system $C$,
and suppose $G$ is obtained from $\mathcal{I}(F,C)$ by attaching loops at the
vertices that appear in a certain subset $\mathcal{L}\subseteq V(G)$. Then in
the notation used above,%
\[
q(G)=\sum_{P\in\mathcal{P}(F,\mathcal{L})}\left(  \frac{y-1}{x-1}\right)
^{\left\vert P\right\vert -c(F)}(x-1)^{\left\vert V(G)\right\vert -\left\vert
\phi(P,C)\right\vert }.
\]

Considering Definition \ref{llc} and Proposition \ref{lpivot}, we see why
\cite{A} does not mention any invariance properties of $q$ under local
complementation or pivoting: because of label swaps involving $\phi$, the
exponent of $x-1$ associated to $P\in\mathcal{P}(F,\mathcal{L})$ is not
generally the same as the exponent associated to $P_{\lambda}^{v}$. Similarly,
the two-term recursion for $q$ that results from Theorem \ref{recursion} when
we set one label to 0 for each vertex does not appear in \cite{A} because of
the shifting-around of powers of $x-1$ under labeled local complementation.
Some of these complications were handled in \cite{Tw} by manipulating weights,
but no motivation involving interlacement was provided there.

Another interlace polynomial, denoted $Q$, was introduced by Aigner and van
der Holst in \cite{AH}. They showed that $Q$ coincides with Bouchet's
\textquotedblleft Tutte-Martin polynomial\textquotedblright\ $M$ \cite{Bi3,
B5}. In particular, $Q$ has a simple fundamental interpretation for circle
graphs: if $G=\mathcal{I}(F,C)$ then $Q(G)$ is essentially the generating
function that records the sizes of all the circuit partitions of $F$. The
description in \cite{AH} makes it clear that $Q$ is obtained from $Q_{\lambda
}$ by using $\phi,\chi,\psi\equiv1$. As the labels are all the same, Theorem
\ref{invariance1} applies in this case; it tells us that $Q$ is invariant
under simple local complementation, as in Corollary 4 of \cite{AH}.

Observe that the fundamental interpretation of $q$ for circle graphs is quite
different from those of $q_{N}$ and $Q$: only $q$ has a fundamental
interpretation as a generating function that records detailed information
regarding the numbers of transitions of particular types. This explains why
the relationship between interlace polynomials and Tutte-Martin polynomials
discussed in \cite{AH, A2, B5} involves $q_{N}$ and $Q$, but not $q$. The
Tutte-Martin polynomials of isotropic systems cannot describe $q$ because they
involve arbitrary labels.

The last polynomial we mention here is Courcelle's multivariate interlace
polynomial \cite{C}. If $G$ is a looped graph with $n$ vertices then $C(G)$ is
a polynomial in $2n+2$ independent indeterminates given by
\[
C(G)=\sum_{\substack{A,B\subseteq V(G)\\A\cap B=\emptyset}}\left(  \prod_{a\in
A}x_{a}\right)  \left(  \prod_{b\in B}y_{b}\right)  u^{\left\vert A\cup
B\right\vert -\nu((G\nabla B)[A\cup B])}v^{\nu((G\nabla B)[A\cup B])},
\]
where $G\nabla B$ denotes the graph obtained from $G$ by toggling loops at the
vertices in $B$. Note that $C(G)$ is essentially a table of the $GF(2)$%
-nullities of all the matrices obtained from adjacency matrices of full
subgraphs of $G$ by toggling some diagonal entries. Consequently $C(G)$
contains the same information as the version of $Q_{\lambda}(G)$ that uses the
indeterminates in $\mathbb{Z}[\{y\}\cup\{\phi_{v},\chi_{v},\psi_{v}$
$\vert$
$v\in V(F)\}]$ as labels. This information is packaged differently in $C(G)$,
using two indeterminates and two possible loop statuses at each vertex rather
than the three labels of $Q_{\lambda}(G)$. The re-packaging obscures the basic
theory of $Q_{\lambda}$ given in Section 3; \cite{C} contains no analogue of
Theorem \ref{invariance1}, and the analogue of Theorem \ref{recursion} is
quite complicated.

\section{Reduction formulas}

As discussed in \cite{A, BH, EMS}, evaluating $q$ or $q_{N}$ is $\#P$-hard in
general. Certainly the same holds for $Q_{\lambda}$, which evaluates to $q$
and $q_{N}$ (with appropriate labels). In contrast, Courcelle \cite{Cj, C}
used techniques of monadic second-order logic to show that computing bounded
portions of his multivariate interlace polynomial $C$ is fixed-parameter
tractable, with clique-width as the parameter. (The restriction to bounded
portions of $C$ is necessitated by the fact that $3^{\left\vert
V(G)\right\vert }$ different products of the indeterminates $x_{a},y_{a}$
appear in $C(G)$.) The techniques of monadic second-order logic apply to a
broad variety of graph polynomials, but they have the compensating
disadvantage of producing algorithms with very large built-in constants.
Consequently it is worth taking the time to investigate special properties of
particular graph polynomials, which may be useful in simplifying computations.
For instance, Bl\"{a}ser and Hoffmann \cite{BH1} have used tree decompositions
and $GF(2)$-nullity calculations to refine Courcelle's result regarding
computation of bounded portions of $C$.

$Q_{\lambda}$ and $C$ determine each other term by term, so the results of
Courcelle, Bl\"{a}ser and Hoffmann apply to $Q_{\lambda}$ too. In this section
we prove a related result using Theorem \ref{comp}, which gives formulas for
the labeled interlace polynomials of graphs that possess split decompositions.
These formulas extend results of Arratia, Bollob\'{a}s and Sorkin \cite{A2}
regarding interlace polynomials of substituted graphs.

\subsection{Pendant-twin reductions}

Ellis-Monaghan and Sarmiento \cite{EMS} showed that the vertex-nullity
interlace polynomial $q_{N}$ can be calculated in polynomial time for
bipartite distance hereditary graphs, i.e. graphs that can be completely
described by two types of pendant-twin reductions \cite{BaM}. Their argument
involved the relationship between the Tutte polynomials of series-parallel
graphs and circuit partitions of the 4-regular graphs that arise as medial
graphs of series-parallel graphs imbedded in the plane \cite{L2, L1, L, Ma}.
We extended this result to the two-variable interlace polynomial $q$ and
general distance hereditary graphs \cite{Tw}, by showing that a
vertex-weighted version of $q$ satisfies reduction formulas which allow for
the consolidation of pendant or twin vertices into a single relabeled vertex;
these reduction formulas are analogous to the series-parallel reductions of
electrical circuit theory. Similar reduction formulas were also used by
Bl\"{a}ser and Hoffman \cite{BH} in their analysis of the complexity of
interlace polynomial computations.

$Q_{\lambda}$ also satisfies pendant-twin reduction formulas, which are of use
in recursive calculations for distance hereditary graphs.

\begin{proposition}
\label{nonadjtwin}Suppose $v$ and $w$ are distinct, nonadjacent vertices of a
labeled simple graph $G$, which have precisely the same neighbors. Then
$Q_{\lambda}(G)=Q_{\lambda}((G-w)^{\prime})$, where $(G-w)^{\prime}$ is
obtained from $G-w$ by changing labels at $v$:
\begin{align*}
\phi_{(G-w)^{\prime}}(v)  &  =\phi_{G}(v)\phi_{G}(w)+\psi_{G}(v)\psi_{G}(w),\\
\chi_{(G-w)^{\prime}}(v)  &  =\phi_{G}(v)\chi_{G}(w)+\chi_{G}(v)\phi
_{G}(w)+\chi_{G}(v)\chi_{G}(w)\cdot y\\
&  +\chi_{G}(v)\psi_{G}(w)+\psi_{G}(v)\chi_{G}(w)\text{, and}\\
\psi_{(G-w)^{\prime}}(v)  &  =\phi_{G}(v)\psi_{G}(w)+\psi_{G}(v)\phi_{G}(w).
\end{align*}

\end{proposition}

\begin{proof}
If $G=\mathcal{I}(F,C)$, the proof is indicated in Figure \ref{dirintf6}: each
configuration of $v$ and $w$ in $F$ gives rise to a corresponding
configuration in $F-w$. In general, we verify that $Q_{\lambda}(G)=Q_{\lambda
}((G-w)^{\prime})$ by checking that each $G_{P}$ matrix obtained by applying
Definition \ref{assograph} to $G$ has the same $GF(2)$-nullity as the
corresponding $G_{P^{\prime}}^{\prime}$ matrix. For the configurations
involving $\phi$ in $G$ this equality is obvious, as the two matrices are
identical. For the other configurations the equality is not quite so obvious.
For instance, if we add the first two rows of the first matrix displayed below
to each row in the set containing $M_{11}$, we conclude that
\[
\nu\left(
\begin{array}
[c]{cccc}%
1 & 0 & \mathbf{1} & \mathbf{0}\\
0 & 1 & \mathbf{1} & \mathbf{0}\\
\mathbf{1} & \mathbf{1} & M_{11} & M_{12}\\
\mathbf{0} & \mathbf{0} & M_{21} & M_{22}%
\end{array}
\right)  =\nu\left(
\begin{array}
[c]{cccc}%
1 & 0 & \mathbf{1} & \mathbf{0}\\
0 & 1 & \mathbf{1} & \mathbf{0}\\
\mathbf{0} & \mathbf{0} & M_{11} & M_{12}\\
\mathbf{0} & \mathbf{0} & M_{21} & M_{22}%
\end{array}
\right)  =\nu\left(
\begin{array}
[c]{cc}%
M_{11} & M_{12}\\
M_{21} & M_{22}%
\end{array}
\right)  .
\]
This explains why $\psi_{G}(v)\psi_{G}(w)$ is included in $\phi_{(G-w)^{\prime
}}(v)$. Similarly,%
\[
\nu\left(
\begin{array}
[c]{cccc}%
0 & 0 & \mathbf{1} & \mathbf{0}\\
0 & 0 & \mathbf{1} & \mathbf{0}\\
\mathbf{1} & \mathbf{1} & M_{11} & M_{12}\\
\mathbf{0} & \mathbf{0} & M_{21} & M_{22}%
\end{array}
\right)  =\nu\left(
\begin{array}
[c]{cccc}%
0 & 0 & \mathbf{0} & \mathbf{0}\\
0 & 0 & \mathbf{1} & \mathbf{0}\\
\mathbf{0} & \mathbf{1} & M_{11} & M_{12}\\
\mathbf{0} & \mathbf{0} & M_{21} & M_{22}%
\end{array}
\right)  =1+\nu\left(
\begin{array}
[c]{ccc}%
0 & \mathbf{1} & \mathbf{0}\\
\mathbf{1} & M_{11} & M_{12}\\
\mathbf{0} & M_{21} & M_{22}%
\end{array}
\right)
\]
explains why $\chi_{G}(v)\chi_{G}(w)\cdot y$ is included in $\chi
_{(G-w)^{\prime}}(v)$.
\end{proof}

%

\begin{figure}
[ptb]
\begin{center}
\includegraphics[
trim=0.933447in 4.280900in 1.073629in 0.939188in,
height=4.1373in,
width=4.7063in
]%
{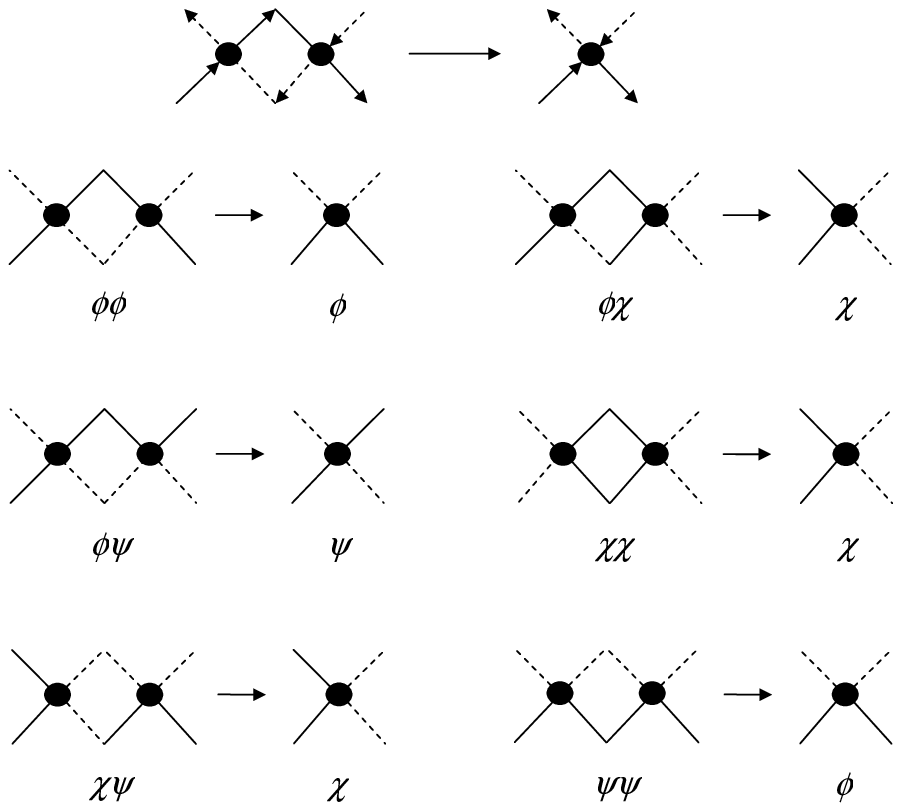}%
\caption{One vertex replaces two vertices that give rise to nonadjacent twins
in the interlacement graph.}%
\label{dirintf6}%
\end{center}
\end{figure}
%

\begin{figure}
[ptb]
\begin{center}
\includegraphics[
trim=0.935921in 8.959723in 1.071980in 0.940257in,
height=0.6253in,
width=4.7072in
]%
{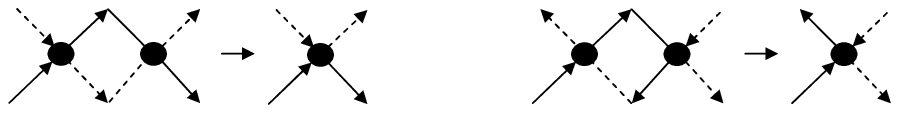}%
\caption{On the left, one vertex replaces two vertices that give rise to
adjacent twins in the interlacement graph. On the right, one vertex replaces
two vertices that give rise to a pendant vertex and its lone neighbor.}%
\label{dirintf7}%
\end{center}
\end{figure}

We leave it to the reader to verify the next two propositions, which give
analogous results for adjacent twins and pendant vertices. For circle graphs,
the propositions may be verified by completing analogues of Figure
\ref{dirintf6} for the two configurations shown in Figure \ref{dirintf7}.

\begin{proposition}
\label{adjtwin}Suppose $v$ and $w$ are neighbors in a labeled simple graph
$G$, which have precisely the same neighbors outside $\{v,w\}$. Then
$Q_{\lambda}(G)=Q_{\lambda}((G-w)^{\prime})$, where $(G-w)^{\prime}$ is
obtained from $G-w$ by changing labels at $v$:
\begin{align*}
\phi_{(G-w)^{\prime}}(v)  &  =\phi_{G}(v)\phi_{G}(w)+\chi_{G}(v)\chi_{G}(w),\\
\chi_{(G-w)^{\prime}}(v)  &  =\phi_{G}(v)\chi_{G}(w)+\chi_{G}(v)\phi
_{G}(w)\text{, and}\\
\psi_{(G-w)^{\prime}}(v)  &  =\phi_{G}(v)\psi_{G}(w)+\chi_{G}(v)\psi
_{G}(w)+\psi_{G}(v)\phi_{G}(w)+\psi_{G}(v)\chi_{G}(w)\\
&  +\psi_{G}(v)\psi_{G}(w)\cdot y.
\end{align*}

\end{proposition}

\begin{proposition}
\label{pendant}Suppose $v$ and $w$ are neighbors in a labeled simple graph
$G$, and $w$ has no neighbor other than $v$. Then $Q_{\lambda}(G)=Q_{\lambda
}((G-w)^{\prime})$, where $(G-w)^{\prime}$ is obtained from $G-w$ by changing
labels at $v$:
\begin{align*}
\phi_{(G-w)^{\prime}}(v)  &  =\phi_{G}(v)\phi_{G}(w)+\chi_{G}(v)\phi
_{G}(w)\cdot y+\chi_{G}(v)\chi_{G}(w)\\
&  +\chi_{G}(v)\psi_{G}(w)+\psi_{G}(v)\phi_{G}(w),\\
\chi_{(G-w)^{\prime}}(v)  &  =\phi_{G}(v)\chi_{G}(w)+\psi_{G}(v)\psi
_{G}(w)\text{, and}\\
\psi_{(G-w)^{\prime}}(v)  &  =\phi_{G}(v)\psi_{G}(w)+\psi_{G}(v)\chi_{G}(w).
\end{align*}

\end{proposition}

For ease of reference we implement the removal of an isolated vertex $w$ in a
similar way, incorporating information about the labels of $w$ in updated
labels for a different vertex $v$.

\begin{proposition}
\label{isolated}Suppose $w$ is an isolated vertex of a labeled simple graph
$G$, i.e., $w$ has no neighbor in $G$. Let $v$ be any other vertex of $G$.
Then $Q_{\lambda}(G)=Q_{\lambda}((G-w)^{\prime})$, where $(G-w)^{\prime}$ is
obtained from $G-w$ by changing labels at $v$:
\begin{align*}
\phi_{(G-w)^{\prime}}(v)  &  =\phi_{G}(v)Q_{\lambda}(\{w\})\text{, }%
\chi_{(G-w)^{\prime}}(v)=\chi_{G}(v)Q_{\lambda}(\{w\})\text{, and }\\
\psi_{(G-w)^{\prime}}(v)  &  =\psi_{G}(v)Q_{\lambda}(\{w\}).
\end{align*}

\end{proposition}

Suppose $G$ can be reduced to a single vertex using Propositions
\ref{nonadjtwin} -- \ref{isolated}. As noted in Corollary 5.3 of \cite{EMS},
such a reduction of $G$ can be found in polynomial time, by searching
repeatedly for isolated vertices, degree-one vertices and pairs of vertices
$v,w$ with the same neighbors outside $\{v,w\}$. In order to recursively
describe the value of $Q_{\lambda}$ for an $R$-labeled version of $G$, we
apply the formulas of the appropriate proposition at each step. As each step
involves removing a vertex, these formulas provide a description of
$Q_{\lambda}(G)$ in polynomial time.

This description may or may not provide a polynomial time computation of
$Q_{\lambda}(G)$, depending on the computational properties of the ring $R$.
For instance if $R$ is a polynomial ring with three indeterminates for each
$v\in V(G)$, then Definition \ref{linterlace} includes contributions from
$3^{\left\vert V(G)\right\vert }$ different products of monomials; it is
impossible to explicitly compute such a large number of terms in polynomial
time. For such a ring, Propositions \ref{nonadjtwin} -- \ref{isolated} are not
really \emph{reductions} in a practical sense; they simply exchange
combinatorial complexity (expressed in the structure of $G$) for algebraic
complexity (expressed in the label formulas). On the other hand, if we are
working over $\mathbb{Z}$ and using $\phi$, $\chi$ and $\psi$ labels that come
from a small set of constants and indeterminates (like those used in obtaining
$q_{N}$, $q$ or $Q$ as instances of $Q_{\lambda}$), then we can determine
$Q_{\lambda}(G)$ by evaluating the indeterminates repeatedly in $\mathbb{Z}$,
and interpolating. Each individual evaluation involves only arithmetic in
$\mathbb{Z}$, which is computationally inexpensive, so for such a ring
Propositions \ref{nonadjtwin} -- \ref{isolated} provide a genuine
polynomial-time computation.

\subsection{Split reductions}

We discuss the following definition only briefly, and refer the reader to
Cunningham \cite{Cu} and Courcelle \cite{Cm, Cj, C} for thorough presentations.

\begin{definition}
\label{join}Let $H$ and $K$ be disjoint simple graphs, each with at least two
vertices. Suppose $S\subseteq V(H)$ and $T\subseteq V(K)$. Then the
\emph{join} of $H$ and $K$ with respect to $S$ and $T$ is the graph
$(H,S)\ast(K,T)=(K,T)\ast(H,S)$ obtained from the union $H\cup K$ by adding
edges connecting all the elements of $S$ to all the elements of $T$. The sets
$V(H)$ and $V(K)$ constitute a \emph{split} of $(H,S)\ast(K,T)$. If $H$ and
$K$ are labeled then the vertices of $(H,S)\ast(K,T)$ inherit labels directly
from $H$ and $K$.
\end{definition}

As an abuse of notation we will find it convenient to use $\ast$ also when $H$
or $K$ has only one vertex. That is, for any simple graph $G$ and any $v\in
V(G)$, we may write $G=(\{v\},\{v\})\ast(G-v,N_{G}(v))$,\ even though $\{v\}$
and $V(G)-\{v\}$ do not constitute a split of $G$.

\begin{definition}
\label{reduction}Let $G$ be a graph with a split $G=(H,S)\ast(K,T)$. Then the
\emph{split reduction} of $G$ with respect to $H$ is the graph obtained from
$K$ by adjoining one new vertex $h$, with open neighborhood $T$. That is, the
split reduction of $(H,S)\ast(K,T)$ with respect to $H$ is $(\{h\},\{h\})\ast
(K,T)$.
\end{definition}

Note that if $G=(H,S)\ast(K,T)$ and $\left\vert V(H)\right\vert =2$ then
either the two vertices of $H$ are twins in $G$, or else one vertex of $H$ is
isolated or pendant in $G$. Propositions \ref{nonadjtwin} -- \ref{isolated}
tell us that in each of these cases, the new vertex $h$ of the split reduction
may be labeled in such a way that the split reduction and $G$ have the same
$Q_{\lambda}$ polynomial. As we will see in Theorem \ref{comp}, the same is
true for every split graph: $(H,S)\ast(K,T)$ shares its $Q_{\lambda}$
polynomial with a reduced graph $(\{h_{S}\},\{h_{S}\})\ast(K,T)$, in which one
appropriately labeled vertex $h_{S}$ replaces the ordered pair $(H,S)$.

To motivate this result, consider a connected 4-regular graph $F$ that has a
\textquotedblleft4-valent\ subgraph\textquotedblright\ $E$. That is, there are
precisely four edges connecting vertices of $E$ to vertices of $F-E$; see the
left-hand side of Figure \ref{dirintf8} for an example. Let $C$ be an Euler
circuit of $F$, and let $S$ (resp. $T$) contain every vertex inside $E$ (resp.
outside $E$) that is encountered exactly once on each passage of $C$ through
$E$ (resp. $F-E$). Then every vertex of $S$ is adjacent to every vertex of $T$
in$\ \mathcal{I}(F,C)$, and these are the only adjacencies connecting vertices
of $E$ to vertices of $F-E$ in $\mathcal{I}(F,C)$. That is, if $H$ and $K$ are
the subgraphs of $\mathcal{I}(F,C)$ induced by $V(E)$ and $V(F)-V(E)$,
respectively, then $\mathcal{I}(F,C)=(H,S)\ast(K,T)$.%
\begin{figure}
[ptb]
\begin{center}
\includegraphics[
trim=2.135714in 8.541475in 2.682424in 0.802267in,
height=1.0464in,
width=2.5988in
]%
{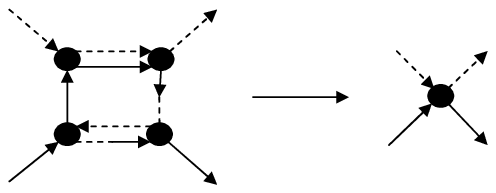}%
\caption{Every interlacement graph arising from the configuration on the left
splits.}%
\label{dirintf8}%
\end{center}
\end{figure}

Observe that in this special case, a circuit partition $P$ of $F$ involves one
of three possible \textquotedblleft whole-$E$ transitions\textquotedblright%
\ that reflect the connections in $P$ involving the four edges connecting $E$
to $F-E$. Comparing these to the connections in $C$, we obtain a
\textquotedblleft whole-$E$ transition label\textquotedblright\ of $P$ with
respect to $C$, corresponding to the sum of all the label products that
represent choices of transitions at the vertices of $E$ that are consistent
with the \textquotedblleft whole-$E$ transition\textquotedblright\ of
$P$.\ These \textquotedblleft whole-$E$ transition labels\textquotedblright%
\ may be used to duplicate $\pi(F,C)$ using circuit partitions of the
simplified graph obtained from $F$ by replacing $E$ with a single vertex.
Equivalently, if we begin a computation of $Q_{\lambda}(G)=\pi(F,C)$ by
applying Theorem \ref{recursion} repeatedly to eliminate the vertices of $E$,
then we can obtain the three \textquotedblleft whole-$E$ transition
labels\textquotedblright\ by collecting terms.\ (This way to structure a
computation -- exhaust an appropriate kind of local substructure, and then
collect terms before proceeding -- was applied to calculations of knot
polynomials by Conway \cite{Co}; he called the 4-valent regions of knot
diagrams \textit{tangles}.)

The following lemma of Balister, Bollob\'{a}s, Cutler and Pebody will be useful.

\begin{lemma}
\label{tri} (Lemma 2 of \cite{BBCS}) Let $M$ be a symmetric matrix with
entries in $GF(2)$, and $\rho$ a row vector. Then two of the three symmetric
matrices%
\[%
\begin{pmatrix}
0 & \rho\\
\rho^{tr} & M
\end{pmatrix}
\text{, }%
\begin{pmatrix}
1 & 0\\
0 & M
\end{pmatrix}
\text{, and }%
\begin{pmatrix}
1 & \rho\\
\rho^{tr} & M
\end{pmatrix}
\text{ }%
\]
have the same $GF(2)$-nullity, and the nullity of the remaining matrix is
greater by 1.
\end{lemma}

Although we will not require it, we might mention a sharper form of Lemma
\ref{tri} proven in \cite{T5}: two of the three matrices actually have the
same nullspace, and the nullspace of the remaining matrix contains the
nullspace shared by the other two.

\begin{corollary}
Suppose $H$ is a labeled simple graph and $S\subseteq V(H)$. Given a labeled
partition $P\in\mathcal{P}_{\lambda}(H)$, let $H_{P}^{S}$ be the graph
obtained from $H_{P}$ by adjoining an unlooped vertex whose neighbors are the
elements of $S\cap V(H_{P})$, and let $H_{P}^{S\ell}$ be the graph obtained
from $H_{P}^{S}$ by attaching a loop at the new vertex. Then two of the three
numbers%
\[
\nu(H_{P}^{S}),\nu(H_{P}),\nu(H_{P}^{S\ell})
\]
are equal, and the third is greater by 1.
\end{corollary}

\begin{proof}
Let $M$ be the adjacency matrix of $H_{P}$; then the second matrix of Lemma
\ref{tri} has the same $GF(2)$-nullity as $M$. If $\rho$ is the row vector
whose $i$th entry is 1 if and only if the corresponding vertex of $H_{P}$ is
an element of $S$ then the first and third matrices of Lemma \ref{tri} are the
adjacency matrices of $H_{P}^{S}$ and $H_{P}^{S\ell}$ (respectively).
\end{proof}

\begin{definition}
\label{labels}Suppose $H$ is a labeled simple graph and $S\subseteq V(H)$. The
\emph{type} of $P\in P_{\lambda}(H)$ (with respect to $S$) is 1, 2 or 3,
according to which of $\nu(H_{P}^{S}),\nu(H_{P}),\nu(H_{P}^{S\ell})$
(respectively) is the largest. The \emph{labels} of $H$ (with respect to $S$)
are the following:
\[
\phi(H,S)=%
{\displaystyle\sum\limits_{\substack{P\in\mathcal{P}_{\lambda}(H)\\\text{of
type 1}}}}
\left(
{\displaystyle\prod_{v\in\phi(P)}}
\phi_{G}(v)\right)  \left(
{\displaystyle\prod_{v\in\chi(P)}}
\chi_{G}(v)\right)  \left(
{\displaystyle\prod_{v\in\psi(P)}}
\psi_{G}(v)\right)  y^{\nu(H_{P})},
\]

\[
\chi(H,S)=%
{\displaystyle\sum\limits_{\substack{P\in\mathcal{P}_{\lambda}(H)\\\text{of
type 2}}}}
\left(
{\displaystyle\prod_{v\in\phi(P)}}
\phi_{G}(v)\right)  \left(
{\displaystyle\prod_{v\in\chi(P)}}
\chi_{G}(v)\right)  \left(
{\displaystyle\prod_{v\in\psi(P)}}
\psi_{G}(v)\right)  y^{\nu(H_{P})-1}~
\]
and%

\[
\psi(H,S)=%
{\displaystyle\sum\limits_{\substack{P\in\mathcal{P}_{\lambda}(H)\\\text{of
type 3}}}}
\left(
{\displaystyle\prod_{v\in\phi(P)}}
\phi_{G}(v)\right)  \left(
{\displaystyle\prod_{v\in\chi(P)}}
\chi_{G}(v)\right)  \left(
{\displaystyle\prod_{v\in\psi(P)}}
\psi_{G}(v)\right)  y^{\nu(H_{P})}.
\]

\end{definition}

Observe that if we use $\phi(H,S)$, $\chi(H,S)$ and $\psi(H,S)$ as labels for
a one-vertex graph $\{h_{S}\}$ then%
\[
Q_{\lambda}(\{h_{S}\})=\phi(H,S)+\chi(H,S)\cdot y+\psi(H,S)=Q_{\lambda}(H).
\]

\begin{lemma}
\label{replace}Suppose $v$ is a vertex of $K$ and either $v\notin T$ or
$N_{K}(v)\neq T-\{v\}$. Then there exist a graph $K^{\prime}$ and a subset
$T^{\prime}\subseteq V(K^{\prime})$ such that $\left\vert V(K^{\prime
})\right\vert <\left\vert V(K)\right\vert $ and $((H,S)\ast(K,T))_{\lambda
}^{v}-v$ can be obtained from $(H,S)\ast(K^{\prime},T^{\prime})$ through some
(possibly empty) sequence of labeled local complementations.
\end{lemma}

\begin{proof}
If $v\notin T$ then $((H,S)\ast(K,T))_{\lambda}^{v}-v=(H,S)\ast(K_{\lambda
}^{v}-v,T)$.

If $v$ is an element of $T$ whose open neighborhood $N_{K}(v)$ is not
$T-\{v\}$, then $((H,S)\ast(K,T))_{\lambda}^{v}-v=\overline{(H,S)}%
\ast(K_{\lambda}^{v}-v,(T-\{v\})\Delta N_{K}(v))$. Here $\Delta$ denotes the
symmetric difference and $\overline{(H,S)}$ denotes the graph obtained from
$H$ by toggling all adjacencies between vertices of $S$, and interchanging
$\chi$ and $\psi$ labels at every vertex of $S$. As $N_{K}(v)\neq T-\{v\}$,
there is some $x\in(T-\{v\})\Delta N_{K}(v)$. Then
\begin{align*}
&  (\overline{(H,S)}\ast(K_{\lambda}^{v}-v,(T-\{v\})\Delta N_{K}%
(v)))_{\lambda}^{x}\\
&  =(H,S)\ast((K_{\lambda}^{v}-v)_{\lambda}^{x},((T-\{v\})\Delta
N_{K}(v))\Delta(N_{K_{\lambda}^{v}}(x)-\{v\})),
\end{align*}
because the labeled local complementation at $x$ restores the internal
structure of $H$.
\end{proof}

\begin{theorem}
\label{comp}Let $H$ and $K$ be simple graphs with labels in $R$, and suppose
$S\subseteq V(H)$ and $T\subseteq V(K)$. Then
\[
Q_{\lambda}((H,S)\ast(K,T))=Q_{\lambda}((\{h_{S}\},\{h_{S}\})\ast(K,T))
\]
where $\{h_{S}\}$ is the one-vertex graph with vertex labels $\phi(H,S)$,
$\chi(H,S)$ and $\psi(H,S)$.
\end{theorem}

That is, $(H,S)\ast(K,T)$ may be reduced to $(\{h_{S}\},\{h_{S}\})\ast(K,T)$
without changing the value of $Q_{\lambda}$, so long as $h_{S}$ carries the
appropriate label values. In the special case $\left\vert V(H)\right\vert =2$,
this reduction is one of the reductions discussed in subsection 7.1.

\begin{proof}
If $T=\emptyset$ then $(H,S)\ast(K,T)$ is the disjoint union of $H$ and $K$,
so
\begin{align*}
Q_{\lambda}((H,S)\ast(K,T))  &  =Q_{\lambda}(H)\cdot Q_{\lambda}(K)\\
&  =Q_{\lambda}(\{h_{S}\})\cdot Q_{\lambda}(K)=Q_{\lambda}((\{h_{S}%
\},\{h_{S}\})\ast(K,T)).
\end{align*}

If $V(K)=T=\{v\}$ then let $(H,S)\ast(K,T)=G$. By definition, $Q_{\lambda}(G)$
is%
\begin{gather*}
\phi(v)\cdot%
{\displaystyle\sum\limits_{P\in\mathcal{P}_{\lambda}(H)}}
\left(
{\displaystyle\prod_{v\in\phi(P)}}
\phi_{G}(v)\right)  \left(
{\displaystyle\prod_{v\in\chi(P)}}
\chi_{G}(v)\right)  \left(
{\displaystyle\prod_{v\in\psi(P)}}
\psi_{G}(v)\right)  y^{\nu(H_{P})}\\
~\\
+\chi(v)\cdot%
{\displaystyle\sum\limits_{P\in\mathcal{P}_{\lambda}(H)}}
\left(
{\displaystyle\prod_{v\in\phi(P)}}
\phi_{G}(v)\right)  \left(
{\displaystyle\prod_{v\in\chi(P)}}
\chi_{G}(v)\right)  \left(
{\displaystyle\prod_{v\in\psi(P)}}
\psi_{G}(v)\right)  y^{\nu(H_{P}^{S})}\\
~\\
+\psi(v)\cdot%
{\displaystyle\sum\limits_{P\in\mathcal{P}_{\lambda}(H)}}
\left(
{\displaystyle\prod_{v\in\phi(P)}}
\phi_{G}(v)\right)  \left(
{\displaystyle\prod_{v\in\chi(P)}}
\chi_{G}(v)\right)  \left(
{\displaystyle\prod_{v\in\psi(P)}}
\psi_{G}(v)\right)  y^{\nu(H_{P}^{S\ell})}%
\end{gather*}

\begin{gather*}
=\phi(v)\cdot(\phi(H,S)+\chi(H,S)\cdot y+\psi(H,S))\\
+\chi(v)\cdot(\phi(H,S)\cdot y+\chi(H,S)+\psi(H,S))\\
+\psi(v)\cdot(\phi(H,S)+\chi(H,S)+\psi(H,S)\cdot y)\\
~\\
=Q_{\lambda}((\{h_{S}\},\{h_{S}\})\ast(K,T)).
\end{gather*}

We proceed by induction on $\left\vert V(K)\right\vert >1$, with
$T\neq\emptyset$. The argument is split into several cases.

Case 1. If $K$ has a connected component $K^{\prime}$ that does not meet $T$,
then $K^{\prime}$ is also a connected component of both $(\{h_{S}%
\},\{h_{S}\})\ast(K,T)$ and $(H,S)\ast(K,T)$, so by induction%
\begin{gather*}
Q_{\lambda}((\{h_{S}\},\{h_{S}\})\ast(K,T))=Q_{\lambda}((\{h_{S}%
\},\{h_{S}\})\ast(K-K^{\prime},T))\cdot Q_{\lambda}(K^{\prime})\\
=Q_{\lambda}((H,S)\ast(K,T)-K^{\prime})\cdot Q_{\lambda}(K^{\prime
})=Q_{\lambda}((H,S)\ast(K,T)).
\end{gather*}

Case 2. Suppose every connected component of $K$ meets $T$, and there is an
edge $vw$ in $K$ with $w\notin T$. We would like to apply the recursive step%
\begin{equation}
Q_{\lambda}(G)=\phi_{G}(v)\cdot Q_{\lambda}(G-v)+\psi_{G}(v)\cdot Q_{\lambda
}((G_{\lambda}^{v})-v)+\chi_{G}(v)\cdot Q_{\lambda}(((G_{\lambda}%
^{w})_{\lambda}^{v})-v) \tag{1}\label{1}%
\end{equation}
of Theorem \ref{recursion} to $v$ and $w$, with $G=(H,S)\ast(K,T)$.

If $v\notin T$ then $G-v=(H,S)\ast(K-v,T)$, $G_{\lambda}^{v}-v=(H,S)\ast
(K_{\lambda}^{v}-v,T)$ and $((G_{\lambda}^{w})_{\lambda}^{v})-v=(H,S)\ast
((K_{\lambda}^{w})_{\lambda}^{v}-v,T)$. These three equalities still hold if
$G$ is replaced by $(\{h_{S}\},\{h_{S}\})\ast(K,T)$ and $(H,S)$ is replaced by
$(\{h_{S}\},\{h_{S}\})$, and the inductive hypothesis applies in each case. We
conclude that $Q_{\lambda}((H,S)\ast(K,T))=Q_{\lambda}((\{h_{S}\},\{h_{S}%
\})\ast(K,T))$.

If $v\in T$ the situation is more complicated, because local complementation
at $v$ changes the structure of $H$. However Lemma \ref{replace} assures us
that each of the three values of $Q_{\lambda}$ in (\ref{1}) is of the form
$Q_{\lambda}((H,S)\ast(K^{\prime},T))$ with $\left\vert V(K^{\prime
})\right\vert <\left\vert V(K)\right\vert $, so once again we may cite the
inductive hypothesis for each summand.

Case 3. Suppose now that every connected component of $K$ meets $T$ and there
is no edge $vw$ in $K$ with $w\notin T$; then $V(K)=T$. If there is an edge
$vw$ in $K$ then we use (\ref{1}) again. This time though we require Lemma
\ref{replace} only for the second term, because the two consecutive local
complementations in $((G_{\lambda}^{w})_{\lambda}^{v})$ have no cumulative
effect on the internal structure of $H$.

Finally, if there is no edge in $K$ then as $V(K)=T$, $N_{(H,S)\ast
(K,T)}(v)=S$ for every $v\in V(K)$. Consequently the vertices of $K$ are
nonadjacent twins in $(H,S)\ast(K,T)$, and we can consolidate two of them into
a single vertex using the formulas of Proposition \ref{nonadjtwin}.
\end{proof}

The following definition will be helpful in discussing the recursive
implementation of Theorem \ref{comp}.

\begin{definition}
The \emph{split width} of a graph $G$, $sw(G)$, is the largest integer that
satisfies these conditions.

1. $sw(G)\leq\left\vert V(G)\right\vert $.

2. If $G=(H,S)\ast(K,T)$ then $sw(G)\leq\max\{\left\vert V(H)\right\vert
,sw((\{h_{S}\},\{h_{S}\})\ast(K,T))\}$.
\end{definition}

We saw in subsection 7.1 that the reductions of Propositions \ref{nonadjtwin}
-- \ref{isolated} provide a recursive description of $Q_{\lambda}$ for graphs
of split width $\leq2$. In much the same way, Theorem \ref{comp} provides a
recursive description of $Q_{\lambda}$ for graphs of split width $\leq s$, for
each fixed value of the parameter $s$. The outline is simple.

\begin{enumerate}
\item Given a graph $G$ with $sw(G)\leq s$, find a suitable split
$G=(H,S)\ast(K,T)$ by searching for a subgraph $H$ with $\left\vert
V(H)\right\vert \leq s$, whose vertices fall into two subsets: $V(H)-S$ (whose
elements have no neighbors outside $H$) and $S$ (whose elements all have the
same neighbors outside $H$). The number of candidates for $V(H)$ is polynomial
in $n=\left\vert V(G)\right\vert $, because $\left\vert V(H)\right\vert \leq
s$. For each candidate for $V(H)$, there are no more than $2^{s}$ candidates
for $S$.

\item Use the formulas of\ Definition \ref{labels} to determine the labels
$\phi(H,S)$, $\chi(H,S)$, and $\psi(H,S)$. These calculations involve finding
the $GF(2)$-nullities of no more than $3^{s+1}$ different $GF(2)$-matrices,
with each matrix no larger than $(s+1)\times(s+1)$.

\item Proceed to calculate $Q_{\lambda}$ for the reduced graph $(\{h_{S}%
\},\{h_{S}\})\ast(K,T)$.
\end{enumerate}

As in subsection 7.1, the computational complexity of this recursive
description depends on the nature of $R$. In a polynomial ring with three
independent indeterminates for each vertex, the number of operations required
to compute $Q_{\lambda}(G)$ is clearly exponential in $n$, and the recursive
description provides a polynomial-time computation only for bounded portions
of $Q_{\lambda}(G)$.\ As $sw(G)\leq s$ implies a bound on the clique-width of
$G$ (see Proposition 4.16 of \cite{Cm}), this analysis is similar to
Courcelle's result regarding computation of bounded portions of $C$ for graphs
of bounded clique-width \cite{C}.

In $\mathbb{Z}$ or $\mathbb{Q}$, instead, arithmetic is computationally
inexpensive, and we deduce the following theorem. Bl\"{a}ser and Hoffmann
\cite{BH1} have proven a similar result, regarding evaluation of $C$ for
graphs of bounded treewidth.

\begin{theorem}
If $G$ is a $\mathbb{Q}$-labeled simple graph then the problem of evaluating
$Q_{\lambda}$ in $\mathbb{Q}$ is fixed parameter tractable, with split width
as parameter.
\end{theorem}

Polynomials like $q_{N}$, $q$ and $Q$, which are evaluations of $Q_{\lambda}$
in $\mathbb{Z}[x]$ or $\mathbb{Z}[x,y]$ rather than $\mathbb{Q}$, can be
determined by evaluating repeatedly in $\mathbb{Q}$, and then interpolating.

\section{A closing comment}

Many different labeled interlace polynomials are obtained by using different
systems of labels and values of $y$ in $Q_{\lambda}$. At one extreme, the
polynomials contain very little information. For instance using $\phi
,\chi,\psi\equiv0$ yields $Q_{\lambda}(G)=0$, while using $y=1$ and $\phi
,\chi,\psi\equiv1$ yields $Q_{\lambda}(G)=3^{\left\vert V(G)\right\vert }$. At
the other extreme, if the elements of $\{y\}\cup\{\phi(v)$, $\chi(v)$,
$\psi(v)\mid v\in V(G)\}$ are independent variables then $Q_{\lambda}(G)$
contains enough information to determine a looped, simple graph $G$ up to
isomorphism. Indeed, $G$ is determined up to isomorphism even with $y=0$ and
$\psi\equiv0$, so long as independent variables are used for the elements of
$\{\phi(v)$, $\chi(v)\mid v\in V(G)\}$. Much remains to be discovered
regarding the significance of labeled interlace polynomials that fall between
these extremes.

\end{document}